\documentclass[reqno]{amsart}
\usepackage{amssymb}
\usepackage{amsmath}
\usepackage{amsthm}
\usepackage[usenames]{color}
\usepackage{graphicx}
\usepackage{cite}

\usepackage[colorlinks,linkcolor=blue,anchorcolor=red,citecolor=blue]{hyperref}
\usepackage{a4wide}
\usepackage{marginnote}

\usepackage{color}

\newtheorem{lemma}{Lemma}[section]
\newtheorem{theorem}{Theorem}[section]

\newtheorem{proposition}{Proposition}[section]
\newtheorem{remark}{Remark}[section]

\numberwithin{equation}{section}

\newcommand{\dis}{\displaystyle}

\newcommand{\R}{\mathbb{R}}

\renewcommand{\S}{\mathbb{S}}

\newcommand{\CB}{\mathcal{B}}

\newcommand{\CE}{\mathcal{E}}

\newcommand{\CK}{\mathcal{K}}
\newcommand{\CL}{\mathcal{L}}

\newcommand{\ep}{\epsilon}

\newcommand{\na}{\nabla}

\newcommand{\ga}{\gamma}
\newcommand{\om}{\omega}
\newcommand{\la}{\lambda}

\newcommand{\pa}{\partial}

\newcommand{\eps}{\epsilon}

\newcommand{\vertiii}[1]{{\left\vert\kern-0.25ex\left\vert\kern-0.25ex\left\vert #1
    \right\vert\kern-0.25ex\right\vert\kern-0.25ex\right\vert}}

\makeatletter
\@namedef{subjclassname@2020}{%
	\textup{2020} Mathematics Subject Classification}
\makeatother

\begin{document}
\title[Solution with polynomial tail for the Boltzmann equation]{Global mild solution with polynomial tail for the Boltzmann equation in the whole space}

\author[R.-J. Duan]{Renjun Duan}
\address[RJD]{Department of Mathematics, The Chinese University of Hong Kong,
	Shatin, Hong Kong, P.R.~China}
\email{rjduan@math.cuhk.edu.hk}

\author[Z.-G. Li]{Zongguang Li}
	\address[ZGL]{Department of Applied Mathematics, The Hong Kong Polytechnic University, Hung Hom, Hong Kong, P.R.~China}
	\email{zongguang.li@polyu.edu.hk}

\author[S.-Q. Liu]{Shuangqian Liu}
\address[SQL]{School of Mathematics and Statistics and Hubei Key Laboratory of Mathematical Sciences, Central China Normal University, Wuhan 430079, P.R.~China}
\email{sqliu@ccnu.edu.cn}
	
\begin{abstract}
We are concerned with the Cauchy problem on the Boltzmann equation in the whole space. The goal is to construct global-in-time bounded mild solutions near Maxwellians with the perturbation admitting a polynomial tail in large velocities. The full range of both hard and soft potentials with cutoff can be covered. The proof is based on a decomposition of the Boltzmann equation motivated by Caflisch that can capture the time-decay property of solutions  in the way that the slow velocity-decay part decays in time much faster than the fast velocity-decay part. This provides the first result on the construction of slow-decaying solutions when the spatial domain is unbounded.
\end{abstract}

	%\date{\today}
	
	
	%35Q20  	Boltzmann equations
	%35B35  	Stability in context of PDEs

	\maketitle
	\thispagestyle{empty}
	
	\tableofcontents

\section{Introduction}
In the paper, we consider the following Cauchy problem on the spatially inhomogeneous Boltzmann equation in the whole space
\begin{eqnarray}\label{BE}
	&\dis \pa_tF+v\cdot \na_x F=Q(F,F),   \quad &\dis F(0,x,v)=F_0(x,v),
\end{eqnarray}
where the unknown $F(t,x,v)\geq0$ stands for the density distribution function of rarefied gas particles with velocity $v\in \R^3$ at time $t> 0$ and position $x\in \R^3$, and initial data $F_0(x,v)\geq 0$ is given. The Boltzmann collision operator $Q(\cdot,\cdot)$ is bilinear and acts only on velocity variable, given by
\begin{align}\label{defQ}
	Q(G,F)(v)&=\int_{\R^3}\int_{\S^2}B(v-u,\theta)\left[ G(u')F(v')-G(u)F(v)\right]d\omega du\notag\\&:=Q_+(G,F)(v)-Q_-(G,F)(v),
\end{align}
where the post-collision velocities $v'$ and $u'$ satisfy
\begin{align}%\label{velocity}
	\begin{split}
		v'=v-\left[(v-u)\cdot \omega \right]\omega, \quad &u'=u+\left[(v-u)\cdot \omega \right]\omega,\notag
%		u'+v'=u+v,\quad |u'|^2&+|v'|^2=|u|^2+|v|^2.
	\end{split}
\end{align}
with $\omega\in \S^2$, and we have denoted $G(u)=G(t,x,u)$ and likewise for $G(u')$ and $F(v')$. Moreover, $B(v-u,\theta)$ is the Boltzmann collision kernel depending only on $|v-u|$ and $\cos \theta=\omega\cdot(v-u)/|v-u|$.
To the end, we suppose the Grad's angular cutoff assumption
\begin{equation}\label{ass.ck}
	B(v-u,\theta)=|v-u|^\gamma b(\theta),
\end{equation}
with
\begin{equation}
\label{b-bdd}
-3<\ga\leq1,\quad 0\leq b(\theta) \leq C|\cos \theta|
\end{equation}
for some positive constant $C$. Through the paper, for simplicity we call hard potentials in case of $0\leq \gamma\leq 1$ and soft potentials in case of $-3<\gamma<0$.

We set the normalized global Maxwellian $\mu$ as
$$
\mu=\mu(v):=(2\pi)^{-3/2}\exp \left( -|v|^2/2 \right).
$$
Under the perturbation near the Maxwellian, we look for solutions in the form of
\begin{align}\label{pert}
	F=\mu+g
\end{align}
for a new unknown function $g=g(t,x,v)$. Substituting \eqref{pert} into \eqref{BE}, we can reformulate the Cauchy problem on the Boltzmann equation in terms of $g$ as
\begin{eqnarray}\label{rbe}
	&\pa_t g+v\cdot \nabla_x g +\mathcal{L} g=Q(g,g),   \quad &\dis g(0,x,v)=g_0(x,v):=F_0(x,v)-\mu(v),
\end{eqnarray}
where $\mathcal{L}$ is the linearized collision operator denoted by
\begin{equation*}
	\mathcal{L}g=- [Q(\mu,g)+Q(g,\mu)].
\end{equation*}
Moreover, $\mathcal{L}$ can be split into
\begin{equation*}
	\CL=\nu -\CK,
\end{equation*}
with
\begin{align}\label{Defnu}
\nu(v)=\int_{\R^3}\int_{\S^2}B(v-u,\theta)\mu(u)d\omega du\sim (1+|v|)^\ga,
\end{align}
and
\begin{align}\label{DefCK}
\CK f(v)=\int_{\R^3}\int_{\S^2}B(v-u,\theta)\left[ \mu(u')f(v')+\mu(v')f(u')-\mu(v)f(u) \right]\,d\omega du.
\end{align}
By integrating along the backward trajectory in \eqref{rbe}, we obtain the mild form
\begin{align}\label{mild}
	\dis g(t,x,v)=&e^{-\nu(v)t}g_0(x-vt,v)+\int_0^t e^{-\nu(v)(t-s)}(\CK g)(s,x-v(t-s),v)\,ds \notag\\
	&+\int_0^t e^{-\nu(v)(t-s)}Q(g,g)(s,x-v(t-s),v)\,ds.
\end{align}

Before stating our main results, we introduce the following notations on norms. Given a function $f=f(x,v)$, the $L^q_vL^p_x$ norm for $1\leq p,q \leq \infty$ is defined by
\begin{align*}
	\|f\|_{L^q_vL^p_x}:= \left\{\int_{\R^3}\left(\int_{\R^3} |f(x,v)|^p dx\right)^\frac{q}{p}dv\right\}^\frac{1}{q},
\end{align*}
with the usual convention for the $L^\infty$ norm in case $p=\infty$ or $q=\infty$.
For brevity, we write $\|f\|_{L^p_{x,v}}=\|f\|_{L^p_vL^p_x}$ for $1\leq p\leq \infty$, in particular, $\|f\|_{L^\infty_{x,v}}=\|f\|_{L^\infty_vL^\infty_x}$ and $\|f\|_{L^2_{x,v}}=\|f\|_{L^2_vL^2_x}$.
Furthermore, we define a polynominal velocity weight function
\begin{align*}%\label{w}
	 w_k=w_k(v)=(1+|v|)^k.
\end{align*}
For a function $f=f(x,v)\in L^\infty_{v} (L^\infty_x\cap L^1_x)$, the velocity weighted $X_{j,k}$ norm is defined as
\begin{align}\label{defX}
	\|f\|_{X_{j,k}}=\|w_kf\|_{L^\infty_{x,v}}+\|w_jf\|_{L^\infty_vL^2_x}+\|w_jf\|_{L^\infty_vL^1_x}.
\end{align}
Similarly, for a function $f=f(x,v)\in L^\infty_v (L^\infty_x\cap L^2_x)$, the velocity weighted $Y_{j,k}$ norm is defined as
\begin{align}\label{defY}
	\|f\|_{Y_{j,k}}=\|w_kf\|_{L^\infty_{x,v}}+\|w_jf\|_{L^\infty_vL^2_x}.
\end{align}

With the above preparations, the main results of the paper are presented as follows. First, we are concerned with the case of hard potentials. We  assume that the initial perturbation is additionally in $L^\infty_vL^1_x$ compared to the previous results to be discussed later.

\begin{theorem}[Hard potentials]\label{hard}
	Let $0\leq\gamma\leq 1$, then there is $k_0>0$ large enough such that for any $k\geq j\geq k_0$, the following holds true.  There are $\ep_0>0$ and $C>0$ such that  if it holds that $F_0(x,v)=\mu(v)+g_0(x,v)\geq 0$ with $g_0\in L^\infty_{v} (L^\infty_x\cap L^1_x)$ satisfying
	\begin{align}\label{smallnesshard}
		\|g_0\|_{X_{j,k}}\leq \ep_0,
	\end{align}
then the Cauchy problem on the Boltzmann equation \eqref{BE} or \eqref{rbe} admits a unique global mild solution $F(t,x,v)=\mu(v)+g(t,x,v)\geq 0$ in the sense of \eqref{mild} with $g\in L^\infty (0,\infty;L^\infty_v (L^\infty_x\cap L^2_x))$ satisfying the estimate
	\begin{align}\label{HE}
		\|g(t)\|_{Y_{j,k}}\leq C(1+t)^{-\frac{3}{4}}	\|g_0\|_{X_{j,k}},
	\end{align}
for all $t\geq0$.
%where $C_{j,k}$ is a constant which depends only on $j$ and $k$.
\end{theorem}

In case of soft potentials, we also obtain a similar result for initial data with extra velocity weights.

\begin{theorem}[Soft potentials]\label{soft}
	Let $-3<\gamma<0$ and $0<\ep\leq\frac{1}{2}$, then there is $k_0>0$ large enough such that for any $k\geq j\geq k_0$, the following holds true. There are $\ep_0$ and $C>0$ such that if $F_0(x,v)=\mu(v)+g_0(x,v)\geq 0$ with $g_0\in L^\infty_{v} (L^\infty_x\cap L^1_x)$ satisfying
	\begin{align}\label{smallnesssoft}
		\|g_0\|_{X_{j+|\ga|,k+|\ga|}}\leq \ep_0,
	\end{align}
then the Cauchy problem on the Boltzmann equation \eqref{BE} or \eqref{rbe} admits a unique global mild solution $F(t,x,v)=\mu(v)+g(t,x,v)\geq 0$ in the sense of \eqref{mild} with $g\in L^\infty (0,\infty;L^\infty_v (L^\infty_x\cap L^2_x))$ satisfying the estimate	
\begin{align}\label{SE}
		\|g(t)\|_{Y_{j,k}}\leq C(1+t)^{-\frac{3}{4}+\ep}	\|g_0\|_{X_{j+|\ga|,k+|\ga|}},
	\end{align}
for all $t\geq0$.  Note that $C$ may depend on $\ep$ such that  $C$ tends to $\infty$ as $\ep\to 0$.
%where $C_{\ep,j,k}$ is a constant which only depends on $\ep$, $j$ and $k$.
\end{theorem}

There is extensive literature on the global well-posedness theory for the spatially inhomogeneous Boltzmann equation; see \cite{Cer, CIP, Gl} and references therein. In particular, the global existence of solutions with uniqueness and large time behavior has been well understood in the framework of perturbations around global Maxwellians since Ukai \cite{Uk} obtained the first result for hard potentials with angular cutoff. We mention further important progress in different aspects, including \cite{Guo-soft, SS, SG} for soft potentials, \cite{GuoY} for general bounded domains, \cite{AMUXY11,AMUXY12,GS} for non-cutoff potentials, and many others. In those works, the solutions are constructed as the form $F=\mu+\sqrt{\mu} f$ with $f$ belonging to a Banach space that can be embedded into $L^\infty_{x,v}$ so that the perturbation $F-\mu$ decays in large velocities with an exponential rate.

A recent remarkable result was obtained by Gualdani-Mischler-Mouhot \cite{GMM} on the spectral gap of the linearized problem after taking the perturbation as $F=\mu+g$ with $g$ allowed to have only an algebraic large velocity tail. The results in  \cite{GMM} are applicable to the cutoff Boltzmann equation with hard potentials when the space variable is in a torus. We also refer to two earlier papers \cite{MiMo, Mo} and their references. These motivated a lot of recent significant progress of the topic on slow decaying solutions to some kinetic equations. Among them we mention Carrapatoso-Mischler \cite{CM} for the nonlinear Landau equation with Coulomb potentials, Briant-Guo \cite{BG} for the cutoff Boltzmann equation in general bounded domains, Alonso-Morimoto-Sun-Yang \cite{AMSY,AMSY-bd} for the non-cutoff Boltzmann equation with hard potentials, Cao-He-Ji \cite{CHJ} and Cao \cite{Cao-22} for the non-cutoff Boltzmann equation with soft potentials, and Cao \cite{Cao} for the cutoff Boltzmann equation with soft potentials.

All of aforementioned existing results on global existence of solutions with polynomial tail in large velocities concern the case where the space domain is either a torus or a bounded domain. In such case, one can obtain the exponential rate (or sub-exponential rate in soft potentials, cf.~\cite{Cao}) for the linearzied operator $\CB:=-v\cdot \nabla_x-\CL$:
\begin{equation}
\label{int.sg}
\|e^{-\CB t}g_0\|_{L^2_k}\leq C e^{-\lambda t}\|g_0\|_{L^2_k}
\end{equation}
for some constant $\lambda>0$, where the norm of $L^2_k$ is defined as $\|f\|_{L^2_k}=\|w_kf\|_{L^2_{x,v}}$ and $g_0$ is purely microscopic. Furthermore, the above fast time-decay property is helpful to define the dissipative norm
\begin{equation}
\label{int.tn}
\vertiii{g}_\CE=\eta \|g\|_{\CE}+\int_0^\infty \|e^{-\CB t}g\|_\CE dt
\end{equation}
with a small constant $\eta>0$ and an energy norm $\CE$ with polynomial velocity weight, and also the exponential time decay of solutions to the linearized problem can be correspondingly recovered in terms of the dissipative norm $\vertiii{\cdot}_\CE$.

However, it is unclear for us to apply the same argument above to treat the case where the space domain is the whole space. In the case of the whole space, one loses the spectral gap as in \eqref{int.sg} and thus it may fail to directly introduce the same dissipative norm as in \eqref{int.tn} in order to bound the nonlinear dynamics. In fact, it is even not straightfoward to obtain a polynomial time-decay rate for $e^{-\CB t}$ by imposing an extra space integrability on initial data $g_0$ such as $g_0\in L^2_vL^1_x$,  although it is doable in case of the symmetric perturbation for $g=\sqrt{\mu}f$ with the velocity weight $\sqrt{\mu}$, for instance, see Proposition \ref{propL2} for hard potentials and Proposition \ref{propL2s} for soft potentials. For the proof of these two propositions, we refer to \cite{DSa,SS} and referenced therein.

In the paper we expect to develop a concise proof to construct global solutions with polynomial tail in large velocities for the Boltzmann equation in the whole space. The results can be regarded as an extension of \cite{Guo-IUMJ} and \cite{UT} to the situation where the initial data is allowed to admit the polynomial tail in large velocities. Our proof relies on the Caflisch's decomposition \cite{Caf} and the Guo's $L^2-L^\infty$ interplay technique \cite{GuoY,Guo}. Applications of such approach have been carried out in the recent studies \cite{DL-arma,DL-cmaa,DL-hard} on the  Boltzmann equation for uniform shear flow; similar applications also have been made in much earlier works \cite{Br, BG}. In what follows, we explain some key points in the proof of our main results Theorems \ref{hard} and \ref{soft} by considering only the linear problem:
\begin{equation}
\label{int.lp}
\pa_t g+v\cdot \na_x g+\CL g=0,\quad g(0,x,v)=g_0(x,v).
\end{equation}
Recall $\CL=\nu-\CK$. We introduce the decomposition $g=g_1+\mu^{1/2}g_2$ with
\begin{align*}
%\label{}
    \pa_t g_1+v\cdot\na_x g_1 +\nu g_1&=\CK_s g_1,   \\
    \pa_t g_2+v\cdot \na_x g_2+Lg_2 &=\CK_b g_1,
\end{align*}
supplemented with initial data
\begin{equation*}
%\label{ }
g_1(0,x,v)=g_{10}:=g_0,\quad g_2(0,x,v)=0,
\end{equation*}
where $\CK=\CK_s +\mu^{1/2}\CK_b$ with
$$
\CK_s:=\chi_{\{|v|\geq M\}}\CK,\quad \CK_b:=\chi_{\{|v|<M\}}\mu^{-1/2}\CK
$$
for a constant $M>0$ large enough, and $L=\mu^{-1/2}\CL\mu^{1/2}$ is the usual self-adjoint operator on $L^2 (\R^3_v)$. We may think of $\CK_s$ to be small and $\CK_b$ to be bounded in a certain sense. Therefore, the dynamics of $g_1$ behaves as a transport equation with a relaxation for initial data $g_0$ and the dynamics of $g_2$ is dominated by the semigroup $e^{-(v\cdot \na_x+L)t}$ with an inhomogeneous term for zero initial data.  In $L^2$ setting, the long time dynamics of $g=g_1+\mu^{1/2}g_2$ is expected to be determined by $g_2$ and thus the polynomial time decay $(1+t)^{-3/4}$ can be obtained as long as the norm of $g_1$ in $L^\infty_v L^1_x$ can propagate in time. To treat the nonlinear dynamics, it is necessary to further control the $L^\infty$ norm of solutions even in the linear problem above. In fact, in $L^\infty$ setting, it is again direct to estimate $g_1$ while we can make use of the $L^2-L^\infty$ interplay technique to obtain the $L^\infty$ time-decay in terms of the derived $L^2$ time-decay. In the end, we are devoted to obtaining the long time behavior of solutions to the linear problem \eqref{int.lp}  as
\begin{equation}
%\label{ }
\|e^{-\CB t}g_0\|_{Y_{j,k}}\leq C (1+t)^{-\frac{3}{4}}\|g_0\|_{X_{j,k}},\notag
\end{equation}
which is an analogue of \eqref{HE} for the nonlinear problem in case of hard potentials.

For soft potentials $\gamma<0$, the collision frequency $\nu\sim (1+|v|)^\gamma$ is degenerate in large velocities, so the estimates would be more subtle. In particular, for $L=\nu-K$, we need to use the decomposition $K=K^m+K^c$ that was first introduced in \cite{Guo-soft} for dealing with the soft potentials, and hence we can follow the iteration method in \cite{DHWY} to carry out the $L^2-L^\infty$ interplay estimates in soft potentials. Moreover, it is also delicate to obtain the time-decay of $g_1$. In general, one can get the polynomial time decay of any order by imposing much higher order polynomial velocity weight on initial data. However, it seems impossible to do so for treating the source term $\CK_s g_1$ in the weighted sense. Thus, in terms of the elementary inequality
 \begin{equation*}
%\label{ }
\int^t_0e^{-\nu(v)(t-s)}\nu(v)(1+s)^{-r}ds\leq C_r(1+t)^{-\min\{r,1\}}\qquad \text{for}\ -3<\ga<0,
\end{equation*}
cf.~Lemma \ref{ledecay}, we are forced to take $r=1-\epsilon<1$ such that one can close the estimates on $g_1$ with the same time weight function $(1+t)^{1-\epsilon}$ on both sides of the equation.

Two remarks are listed below.

\begin{remark}
We refer to Theorem \ref{local} for the local-in-time existence of \eqref{BE} with initial data $F_0(x,v):=\mu+g_0(x,v)\geq 0$ satisfying
\begin{equation*}
%\label{ }
\|w_k g_0\|_{L^{\infty}_{x,v}}<\ep_1.
%\notag
\end{equation*}
Thus, from Theorems \ref{hard} and \ref{soft}, no solutions in large time have been found without  requiring smallness of $g_0$ in the velocity weighted norm of $L^\infty_vL^1_x$ additionally. In fact, such extra smallness is an essential assumption in the proof of global existence of solutions since we need to use the time decay of solutions to close existence estimates. It is unclear whether it is possible  to remove the restriction on the $L^\infty_vL^1_x$ norm for initial data in both Theorems \ref{hard} and \ref{soft}. The situation is different from that in case of the symmetric perturbation (cf.~\cite{UT}, for instance) where the global existence can be established in the pure $L^2\cap L^\infty$ setting. We will clarify this issue in another work.
\end{remark}

\begin{remark}
Regarding the time decay of solutions, the rate $(1+t)^{-3/4}$ in \eqref{HE} is optimal for initial data in $L^\infty_vL^1_x$ as it corresponds to the rate of the heat semigroup or linearized Navier-Stokes in the whole space. Though, it is still a problem to remove the small parameter $\epsilon$ in \eqref{SE} to recover an optimal rate for soft potentials. On the other hand, for both hard and soft potentials, the optimal rate for $\|w_kg(t)\|_{L^\infty_{x,v}}$ should be $(1+t)^{-3/2}$. For simplicity, we would not pursuit the further proof of such issue in the current work.
\end{remark}

\medskip
\noindent{\bf Added Note:}\,The current work was uploaded to arXiv (arXiv:2212.04676) on December 9, 2022. Several relevant studies also have emerged, including \cite{CG} and \cite{CDL}, which address slow velocity-decaying solutions to the non-cutoff Boltzmann equation in the whole space, and \cite{LWW}, which investigates Green's functions of the cutoff Boltzmann equation under slow velocity decay conditions. One of our motivations is to further develop the inclusion of physical boundary effects in unbounded domains, as seen in \cite{CDZ} for the cutoff Boltzmann equation in an infinite layer. Additionally, this work offers the potential to construct time-dependent kinetic planar Couette flow solutions to the Boltzmann equation in an infinite channel, contrasting with the recent results in \cite{DLY} and \cite{DLSY} for finite channel domains. 

\bigskip
The rest of this paper is organized as follows. In section \ref{sec2} we collect preliminary estimates and inequalities. In section \ref{sec3} we provide a complete proof of the local-in-time existence of nonnegative mild solutions. In section \ref{sec4} and section \ref{sec5} we give the proof of the main results Theorem \ref{hard} and Theorem \ref{soft}, respectively.

%\newpage
\section{Preliminaries}\label{sec2}
It is important to study the properties of the operator $\CK$ defined in \eqref{DefCK}. In the region where the velocity $v$ is large enough, $\CK$  can be sufficiently small. Such issue on smallness of $\CK$ in the large velocity region was initiated from \cite{AEP-87} and extensively developed in many other works later. For the purpose, we first borrow the following result from a recent work \cite[Lemma 2.11 and Lemma 6.2]{Cao}.

\begin{lemma}\label{leQ}
Let $-3<\ga\leq 1$. For any $k>\max\{3,3+\ga\}$, there is a constant $C_k>0$ depending only on $k$ such that
	\begin{align}\label{controlK}
	\int_{\R^3}\int_{\S^2}|v-u|^\ga\frac{w_k(v)}{w_k(v')}e^{-\frac{|u'|^2}{2}}d\omega du\leq \frac{C}{k^{\frac{\ga+3}{4}}}\nu(v)+C_k\frac{\nu(v)}{(1+|v|)^2},\quad \forall\,v\in \R^3,
\end{align}	
and
	\begin{align}\label{controlQ}
		\int_{\R^3}\int_{\S^2}|v-u|^\ga\frac{w_k(v)}{w_k(v')w_k(u')}d\omega du\leq C_k\nu(v), \quad \forall\,v\in \R^3,
	\end{align}
where $C>0$ is a universal constant independent of $k$. Morever, when $0\leq \ga\leq 1$,  the bound in \eqref{controlQ} can be refined as the one similar to \eqref{controlK}, that is,
	\begin{align}\label{controlQ2}
		\int_{\R^3}\int_{\S^2}|v-u|^\ga\frac{w_k(v)}{w_k(v')w_k(u')}d\omega du\leq \frac{C}{k}\nu(v)+C_k\frac{\nu(v)}{(1+|v|)^2}, \quad \forall\,v\in \R^3.
	\end{align}
\end{lemma}

Using Lemma \ref{leQ}, we are able to specify the smallness property of $\CK f(v)$ over the large velocity region in the following lemma. The similar results also have been studied in \cite{DL-arma,DL-cmaa,DL-hard}.

\begin{lemma}\label{leKs}
	Let $-3< \ga \leq 1$. For any $k>\max\{3,3+\ga\}$, there is a constant $C_k>0$ depending only $k$ such that for any $M>0$,
	\begin{align}\label{smallCK}
		\chi_{\{|v|>M\}}w_k(v)|\CK f(v)|\leq \left(\frac{C}{k^{\frac{\ga+3}{4}}}+\frac{C_k}{M^2}\right)\nu(v)	\|w_kf\|_{L^\infty_v}, \quad \forall\,v\in \R^3,
	\end{align}
where $C>0$ is a universal constant independent of $k$ and $M$.
\end{lemma}

\begin{proof}
	From the definition of $\CK$ \eqref{DefCK} as well as the assumptions \eqref{ass.ck} and \eqref{b-bdd}, we have
	\begin{align*}
	&\chi_{\{|v|>M\}}w_k(v)|\CK f(v)|\\
	&\leq C\|w_kf\|_{L^\infty_v} \int_{\R^3}\int_{\S^2}\chi_{\{|v|>M\}}|v-u|^\gamma\left( \frac{w_k(v)\mu(u')}{w_k(v')}+\frac{w_k(v)\mu(v')}{w_k(u')}+\frac{w_k(v)\mu(v)}{w_k(u)}\right)d\omega du\notag\\
		&=C\|w_kf\|_{L^\infty_v} \left(H_1(v)+H_2(v)+H_3(v)\right),
	\end{align*}
for any $v\in \R^3$, where $H_i(v)$ $(i=1,2,3)$ denote the corresponding integral terms.
For $H_1(v)$, using \eqref{controlK}, it follows that
\begin{align*}
	H_1(v)\leq\left(\frac{C}{k^{\frac{\ga+3}{4}}}+\frac{C_k}{M^2}\right)\nu(v).
\end{align*}
For $H_2(v)$, the same estimate as above holds true since we can exchange $v'$ and $u'$ by rotation, cf.~\cite{Gl}. For $H_3(v)$, we note that for any $k>\max\{3,3+\ga\}$ there is a constant $C_k$ such that  $\int_{\R^3}|v-u|^\ga (1+|u|)^{-k}du\leq C_k\nu(v)$ and hence it holds
\begin{align*}
	H_3(v)\leq \chi_{\{|v|>M\}}w_k(v)\mu(v)\cdot C_k\nu(v)\leq \chi_{\{|v|>M\}}\frac{1}{|v|^2}\cdot C_k\nu(v)\leq \frac{C_k}{M^2}\nu(v).
\end{align*}
%Without loss of generality, one can get $$H_3(v)\leq \frac{C_0}{M^2}\nu(v)$$ by choosing $C_0$ to be large enough.
Collecting the estimates for $H_1(v)$, $H_2(v)$ and $H_3(v)$, we obtain \eqref{smallCK} and complete the proof of Lemma \ref{leKs}.
\end{proof}

On the other hand, in the symmetric framework, the linear operator $L$ defined as
\begin{align}\label{def.L}
	Lf:=\frac{1}{\sqrt{\mu}}\CK (\sqrt{\mu}f)=-\frac{1}{\sqrt{\mu}}Q(\mu,\sqrt{\mu}f)-\frac{1}{\sqrt{\mu}}Q(\sqrt{\mu}f,\mu),
\end{align}
is self-adjoint and non-negative definite on $L^2_v$, and thus it has much better properties than the operator $\CK$. In fact, it is well known that we have the Grad's splitting $L=\nu-K$ under  the assumptions \eqref{ass.ck} and \eqref{b-bdd}. Here, the multiplication operator $\nu$ satisfies \eqref{Defnu} and the integral operator $K$ satisfies the following property; see \cite{DHWY} and references therein for more details.

\begin{lemma}\label{propK}
Let $-3< \ga \leq 1$. There is a kernel function $k(\cdot,\cdot)$ with $k(v,\eta)=k(\eta,v)$ such that
%there exists a function $k(v,\eta)=k(\eta,v)$ such that
	\begin{align*}%\label{RepK}
		Kf(v)=\int_{\R^3}k(v,\eta)f(\eta)\,d\eta,
		%\quad v\in \R^3,
	\end{align*}
	with
	\begin{align*}
	\dis |k(v,\eta)|\leq C|v-\eta|^\ga e^{-\frac{|v|^2}{4}}e^{-\frac{|\eta|^2}{4}}+\frac{C}{|v-\eta|^\frac{3-\ga}{2}}e^{-\frac{|v-\eta|^2}{8}}e^{-\frac{\left| |v|^2-|\eta|^2\right|^2}{8|v-\eta|^2}}.
	\end{align*}
	%where $C_{k}$ is a constant which depends only on $k$.
	Moreover, for any $k\in \R$, there is a constant $C_k$ such that %we have
	\begin{align}\label{Prok}
		\int_{\R^3}\left| k(v,\eta)\right|\frac{w_k(v)}{w_k(\eta)} e^{\frac{|v-\eta|^2}{20}}\,d\eta \leq C_{k}(1+|v|)^{-1},\quad \forall\,v\in \R^3.
	\end{align}
\end{lemma}

The large velocity decay property in \eqref{Prok} is not enough for the use in soft potentials. In fact, \cite{Guo-soft,SG} introduced a further useful decomposition of $K$ via a cut-off smooth function $\chi_m=\chi_m(\tau)$  with a parameter $0\leq m\leq 1$ such that $0\leq \chi_m \leq 1$, and $\chi_m(\tau)=1$ for $\tau \leq m$ and $\chi_m(\tau)=0$ for $\tau \geq 2m$. One can split $K$ into two parts $K=K^m+K^c$ with
\begin{align}\label{defKm}
	K^mf(v)=&\int_{\R^3}\int_{\S^2}B(v-u,\theta)\chi_m(|v-u|)\sqrt{\mu(u)}\notag\\
	&\cdot\left( \sqrt{\mu(u')}f(v')+\sqrt{\mu(v')}f(u')
	-\sqrt{\mu(v)}f(u) \right)d\omega du,
\end{align}
and $K^c=K-K^m$. The following lemma provides useful estimates on both parts $K^m$ and $K^c$. For the proof, we may refer to \cite{Guo-soft,SG} as well as \cite{DHWY}.

\begin{lemma}%\label{Kc}
	Let $-3< \ga < 0$ and $0<m\leq 1$. For the part $K^m$, there is a constant $C$ independent of $m$ such that
	% and $k \in \R$. $K^m$ satisfies
	\begin{align}\label{smallKm}
		%\begin{split}
		|K^mf(v)|\leq Cm^{\ga+3}e^{-\frac{|v|^2}{10}}\|f\|_{L^\infty_v},\quad \forall\,v\in \R^3. %\label{BKm}
		%\end{split}
	\end{align}
For the part $K^c$, there is a kernel function $l(\cdot,\cdot)$ with $l(v,\eta)=l(\eta,v)$ such that
	%Moreover, there is a function $l(v,\eta)=l(\eta,v)$ such that
	\begin{align}\label{RepKc}
		K^cf(v)=\int_{\R^3}l(v,\eta)f(\eta)d\eta,
	\end{align}
	with
	\begin{align*}%\label{Estionl}
	\dis |l(v,\eta)|\leq C|v-\eta|^\ga e^{-\frac{|v|^2}{4}}e^{-\frac{|\eta|^2}{4}}+\frac{C_{m}}{|v-\eta|^\frac{3-\ga}{2}}e^{-\frac{|v-\eta|^2}{16}}e^{-\frac{\left| |v|^2-|\eta|^2\right|^2}{16|v-\eta|^2}}.
\end{align*}
Moreover, for $l(v,\eta)$, the following estimates also hold true. For any $k\in \R$, there is a constant $C_{m,k}$ such that %it holds that
\begin{align}
	&\int_{\R^3}\left| l(v,\eta)\right| \frac{w_k(v)}{w_k(\eta)} d\eta \leq C_{m,k} (1+|v|)^{-1},\label{slowl}\\
		&\int_{\R^3}\left| l(v,\eta)\right| \frac{w_k(v)}{w_k(\eta)} e^{-\frac{|\eta|^2}{20}}d\eta \leq C_{m,k} e^{-\frac{|v|^2}{100}}, \notag\\
		&\int_{\R^3}\left| l(v,\eta)\right| \frac{w_k(v)}{w_k(\eta)} e^{\frac{|v-\eta|^2}{20}}d\eta \leq C_{m,k} %m^{\ga-1}
		\frac{\nu(v)}{(1+|v|)^2},\label{decayl}
	\end{align}
	for all $v\in \R^3$.
\end{lemma}

The next lemma is concerned with the time-decay of the time convolution term of the form $e^{-\nu(v)t}\ast_t \nu(v)(1+t)^{-r}$ for $r>0$. In case of hard potentials, such convolution can absorb the large velocity growth of $\nu(v)$ via the time integration, while in case of soft potentials, it still decays in time with a rate $(1+t)^{-\min\{r,1\}}$  even though $\nu(v)$ is degenerate in large velocities.

\begin{lemma}\label{ledecay}
	Let $-3<\ga\leq1$ and $r>0$. There is a constant $C_r>0$ such that in case $0\leq \gamma\leq 1$, it holds
	\begin{align}
		&\int^t_0e^{-\nu(v)(t-s)}\nu(v)(1+s)^{-r}ds\leq C_r(1+t)^{-r}, \quad \forall\,t\geq 0,
		%\qquad\qquad\,\text{for\ $0\leq\ga\leq1$},
		\label{decay}
	\end{align}
	and in case $-3<\gamma<0$ and $r\neq 1$, it holds
	\begin{align}
		\int^t_0e^{-\nu(v)(t-s)}\nu(v)(1+s)^{-r}ds\leq C_r(1+t)^{-\min\{r,1\}},\quad \forall\,t\geq 0.
		\label{decays}
	\end{align}
\end{lemma}
\begin{proof}
	We always consider the integral by splitting the time interval $[0,t]$ into $[0,t/2]$ and $(t/2,t]$. Let $r>0$. We first consider the proof of \eqref{decay} for $0\leq\ga\leq1$. In this case, $\nu(v)$ has a positive lower bound, namely,  $\inf_v\nu(v)\geq \nu_0$ for a constant $\nu_0>0$. For the interval $0\leq s\leq t/2$, we have $t-s\geq -t/2$ and thus
	\begin{align*}
		\int^{t/2}_0e^{-\nu(v)(t-s)}\nu(v)(1+s)^{-r}ds&\leq e^{-\frac{\nu_0}{4}t}\int^t_0e^{-\frac{\nu(v)}{2}(t-s)}\nu(v)ds\notag\\
		&\leq Ce^{-\frac{\nu_0}{4}t}.
	\end{align*}
	For the other time interval $t/2<s\leq t$, it holds that
	\begin{align}\label{t/2,t}
		\int^t_{t/2}e^{-\nu(v)(t-s)}\nu(v)(1+s)^{-r}ds&\leq C(1+t/2)^{-r}\int^t_0e^{-\nu(v)(t-s)}\nu(v)ds\notag\\
		&\leq C_r(1+t)^{-r}.
	\end{align}
Combing both cases, \eqref{decay} is proved.

Next, we consider \eqref{decays} for $-3<\ga<0$. In this case, $\nu(v)\sim (1+|v|)^\gamma$ is uniformly bounded. Moreover, we note that there is a constant $C>0$ such that $e^{-x}\leq \frac{C}{1+x}$ for any $x\geq 0$. Thus, for $0\leq s\leq t/2$, it holds that
\begin{align*}
	\int^{t/2}_0e^{-\nu(v)(t-s)}\nu(v)(1+s)^{-r}ds&\leq C\int^{t/2}_0\frac{1}{1+\nu(v)(t-s)}\nu(v)(1+s)^{-r}ds\notag\\
	&\leq C\int^{t/2}_0(1+t-s)^{-1}(1+s)^{-r}ds\notag\\
	&\leq \frac{C}{(1+t/2)^{-1}}\int^t_0(1+s)^{-r}ds\notag\\
	&\leq C_r(1+t)^{-\min\{r,1\}},
\end{align*}
where we have used the fact that $r\neq 1$.
The other case $t/2\leq s\leq t$ can be estimated in the same way as \eqref{t/2,t}. Thus, \eqref{decays} is proved. The proof of Lemma \ref{ledecay} is complete.
\end{proof}

In the end we make a list of some elementary inequalities related to the time decay of convolution of two time-decay functions.

\begin{lemma}%\label{timedecay}
	Let $r>0$ and $0<q\leq r$, then there is a constant $C_{r,q}>0$ such that for any $t\geq 0$, %we have
	\begin{align}\label{tdeday}
		&\int^t_0(1+t-s)^{-q}(1+s)^{-r}ds\leq \begin{cases}
			C_{r,q}(1+t)^{-q}, \qquad\qquad \qquad  \text{$r>1$},\\
			C_{r,q}(1+t)^{-q}\log (1+t)\qquad\,   \text{$r=1$},\\
			C_{r,q}(1+t)^{-q-r+1}\qquad\qquad\   \text{$r<1$}.
		\end{cases}
	\end{align}
Furthermore, let $r>0$ and $\lambda>0$, then there is a constant $C_{r,\lambda}>0$ such that for any $t\geq 0$, %it holds that
	\begin{align}\label{tdecay1}
		\int^t_0e^{-\lambda(t-s)}(1+s)^{-r}\,ds\leq C_{r,\lambda}(1+t)^{-r}.
	\end{align}
\end{lemma}

%In the next section, we will prove local existence.

%\newpage
\section{Local-in-time Existence}\label{sec3}
In this section, we establish the local-in-time existence of mild solutions to the Cauchy problem \eqref{BE} or equivalently \eqref{rbe}. For completeness, we give the full details of the proof. In order to state the main result, we define the function space
\begin{align*}
%\label{ }
\mathcal{X}_{T,k,j}:=\{g(t,x,v) &\in  L^\infty(0,T; L^\infty_{x,v}\cap L^\infty_vL^2_x): \\
&\sup_{0\leq t\leq T}\left\|w_k g(t)\right\|_{L^\infty_{x,v}}+\sup_{0\leq t\leq T}\left\|w_j g(t)\right\|_{L^\infty_vL^2_x}<\infty\}.
\end{align*}

\begin{theorem}\label{local}
	Assume \eqref{ass.ck} and \eqref{b-bdd}. There is $k_0>0$ large enough such that for any $k\geq j\geq k_0$, the following holds true. There are $\ep_1>0$ and $T_*>0$ such that if $F_0(x,v):=\mu(v)+g_0(x,v)\geq 0$ with
\begin{equation}
\label{thm.l.as}
\|w_k g_0\|_{L^{\infty}_{x,v}}<\ep_1,
\end{equation}	
then the Cauchy problem on the Boltzmann equation \eqref{BE} or \eqref{rbe} admits a unique mild solution $F(t,x,v)=\mu(v)+g(t,x,v)\geq 0$ with $(t,x,v)\in [0,T_*]\times \R^3_x\times \R^3_v$ in the sense of \eqref{mild} such that $g\in \mathcal{X}_{T_*,k,j}$ satisfies the following estimate:
	\begin{equation}\label{LE}
		\sup_{0\leq t\leq T_*}\left\|w_k g(t)\right\|_{L^\infty_{x,v}}+\sup_{0\leq t\leq T_*}\left\|w_j g(t)\right\|_{L^\infty_vL^2_x}\leq 2\big(\left\|w_k g_0\right\|_{L^\infty_{x,v}}+\left\|w_j g_0\right\|_{L^\infty_vL^2_x}\big).
	\end{equation}
\end{theorem}

\begin{proof}
	We first rewrite \eqref{rbe} as
	\begin{align*}%\label{rrbe}
		&\dis\pa_t g+v\cdot \nabla_x g +\nu g+Q_-(g,g)=\CK g+Q_+(g,g) ,  \\
		 &\dis g(0,x,v)=g_0(x,v)=F_0(x,v)-\mu(v),
	\end{align*}
where $\nu$, $\CK$, $Q_-$ and $Q_+$ are defined in \eqref{Defnu}, \eqref{DefCK} and \eqref{defQ}.
Then we construct the approximation sequence $\{g^n\}_{n=0}^\infty$ as in \cite{DHWY,Li}:
\begin{align*}
	\pa_t g^{n+1}+v\cdot \nabla_x g^{n+1} +\nu g^{n+1}+Q_-(g^{n},g^{n+1})=\CK g^{n}+Q_+(g^{n},g^{n}) ,
\end{align*}
with  $g^{n+1}(0,x,v)=g_0(x,v)$ and $g^0(t,x,v)=0$. It is noted that
\begin{align}
	\dis g^{n+1}(t,x,v)=&e^{-{\int^t_0\zeta^n(\tau,x-v(t-\tau),v)d\tau}}g_0(x-vt,v) \notag\\
	&+\int_0^t e^{-{\int^t_s\zeta^n(\tau,x-v(t-\tau),v)d\tau}}(\CK g^n)(s,x-v(t-s),v)ds \notag\\
	&+\int_0^t e^{-{\int^t_s\zeta^n(\tau,x-v(t-\tau),v)d\tau}}Q_+(g^n,g^n)(s,x-v(t-s),v)ds, \label{appro}
\end{align}
where
\begin{align*}%\label{defzeta}
\zeta^n(\tau,y,v)=\int_{\R^3}\int_{\S^2}B(v-u,\theta)\left[ \mu(u)+g^n (\tau,y,u) \right]d\omega du.
\end{align*}
 Defining $F^n=\mu+g^n$, we have that \eqref{appro} is equivalent to
\begin{align*}
	\dis F^{n+1}(t,x,v)=&e^{-{\int^t_0\zeta^n(\tau,x-v(t-\tau),v)d\tau}}F_0(x-vt,v) \notag\\
	&+\int_0^t e^{-{\int^t_s\zeta^n(\tau,x-v(t-\tau),v)d\tau}}Q_+(F^n,F^n)(s,x-v(t-s),v)ds.
\end{align*}
Then it is direct to see that the nonnegativity $F^n\geq0$ follows by induction.

We prove the bound \eqref{LE} for the approximation sequence $\{g^n\}^\infty_{n=0}$ in case of soft and hard potentials separately. First we consider $-3<\ga<0$. We are going to show that there is $T_*>0$ such that for any $n\geq 0$,
\begin{align}
&\sup_{0\leq s\leq T_*}\|w_kg^{n}(s)\|_{L^\infty_{x,v}}\leq 2\|w_kg_0\|_{L^\infty_{x,v}},\label{le.s.inftybd}\\
&\sup_{0\leq s\leq T_*}\|w_jg^{n}(s)\|_{L^\infty_vL^2_x}\leq 2\|w_jg_0\|_{L^\infty_vL^2_x}\label{le.s.L2bd}.
\end{align}
We prove it by induction. It is obvious to hold for $n=0$, since $g^0\equiv 0$. We assume \eqref{le.s.inftybd} and \eqref{le.s.L2bd} for $n\geq 0$.
By \eqref{appro} and the fact that $\zeta^n(\tau,y,v)=\int_{\R^3}\int_{\S^2}B(v-u,\theta)F^n (\tau,y,u) d\omega du\geq 0$, a direct calculation shows that
\begin{align*}
	\dis \left|w_k(v) g^{n+1}(t,x,v)\right|\leq &\left|w_k(v) g_0(x-vt,v)\right|+\int_0^t \left|w_k(v)(\CK g^n)(s,x-v(t-s),v)\right|ds\notag \\
	&+\int_0^t \left|w_k(v)Q_+(g^n,g^n)(s,x-v(t-s),v)\right|ds.
\end{align*}
Applying the definitions of $\CK$ in \eqref{DefCK} and $Q_+$ in \eqref{defQ} yields
\begin{align}\label{conseq}
	&\dis \left|w_k(v) g^{n+1}(t,x,v)\right|\notag\\
	\leq &\left|w_k(v) g_0(x-vt,v)\right|\notag \\+C\int_0^t &\int_{\R^3}\int_{\S^2}\|w_k g^n(s)\|_{L^\infty_{x,v}}|v-u|^\ga\notag \\
	&\qquad\times\big(\frac{w_k(v)}{w_k(v')}e^{-\frac{|u'|^2}{2}}+\frac{w_k(v)}{w_k(u')}e^{-\frac{|v'|^2}{2}}+\frac{w_k(v)}{w_k(u)}e^{-\frac{|v|^2}{2}}\big)d\omega duds\notag \\
	&+C\int_0^t \int_{\R^3}\int_{\S^2}\|w_k g^n(s)\|_{L^\infty_{x,v}}^2|v-u|^\ga\frac{w_k(v)}{w_k(v')w_k(u')}d\omega duds.
\end{align}
Then it follows from \eqref{controlK}, \eqref{controlQ} and \eqref{conseq} that for $0\leq t\leq T_*$ and $k\geq k_0>3$,
\begin{align}\label{boundsoftn}
	&\dis \left|w_k(v) g^{n+1}(t,x,v)\right|\notag\\
	\leq &\|w_kg_0\|_{L^\infty_{x,v}}+C_k\int_0^t \nu(v)\big(\|w_k g^n(s)\|_{L^\infty_{x,v}}+\|w_k g^n(s)\|_{L^\infty_{x,v}}^2\big)ds\notag\\
	\leq &\|w_kg_0\|_{L^\infty_{x,v}}+C_kT_*\big(\sup_{0\leq s\leq T_*}\|w_k g^n(s)\|_{L^\infty_{x,v}}+\sup_{0\leq s\leq T_*}\|w_k g^n(s)\|_{L^\infty_{x,v}}^2\big).
\end{align}
Using the assumption \eqref{le.s.inftybd}, one gets
\begin{align}\label{boundsoft}
	\dis \left|w_k(v) g^{n+1}(t,x,v)\right|\leq \|w_kg_0\|_{L^\infty_{x,v}}+C_kT_*\big(\|w_k g_0\|_{L^\infty_{x,v}}+\|w_k g_0\|_{L^\infty_{x,v}}^2\big).
\end{align}
Similarly, for $k\geq j\geq k_0>3$,
 \begin{align}\label{localL2}
	&\dis \left\|w_j g^{n+1}(t)\right\|_{L^\infty_vL^2_x}\notag \\
	\leq &\left\|w_j g_0\right\|_{L^\infty_vL^2_x}+\int_0^t \left\|w_j(\CK g^n)(s)\right\|_{L^\infty_vL^2_x}ds\notag \\
	&+\int_0^t \left\|w_jQ_+(g^n,g^n)(s)\right\|_{L^\infty_vL^2_x}ds\notag\\
	\leq &\|w_jg_0\|_{L^\infty_vL^2_x}+C_jT_*\sup_{0\leq s\leq T_*}\|w_j g^n(s)\|_{L^\infty_vL^2_x}\big(1+\sup_{0\leq s\leq T_*}\|w_k g^n(s)\|_{L^\infty_{x,v}}\big)\notag\\
	\leq &\|w_jg_0\|_{L^\infty_vL^2_x}+C_jT_*\|w_j g_0\|_{L^\infty_vL^2_x}\big(1+\|w_k g_0\|_{L^\infty_{x,v}}\big).
\end{align}
In the second inequality above, we use \eqref{controlQ} and the Minkowski inequality that
\begin{align*}
\|w_jQ_+(g^n,g^n)(s,v)\|_{L^2_x}&\leq \int_{\R^3}\int_{\S^2}|v-u|^\ga\frac{w_j(v)}{w_j(v')w_j(u')}\|w_j(v) g^n(s,v)\|_{L^2_x}\|w_j(v) g^n(s,v)\|_{L^\infty_x}d\omega du\notag\\
&\leq C_j\nu(v)\sup_{0\leq s\leq T_*}\|w_j g^n(s)\|_{L^\infty_vL^2_x}\sup_{0\leq s\leq T_*}\|w_k g^n(s)\|_{L^\infty_{x,v}}.
\end{align*}
Let
\begin{align*}
	\dis T_*=\frac{1}{6(C_k+C_j)(1+\|w_k g_0\|_{L^\infty_{x,v}})}>0,
\end{align*}
then it follows from \eqref{boundsoft} and \eqref{localL2} that \eqref{le.s.inftybd} and \eqref{le.s.L2bd} are satisfied for $n+1$.
%\begin{align*}
%	&\sup_{0\leq s\leq T_*}\|w_kg^{n+1}(s)\|_{L^\infty_{x,v}}\leq 2\|w_kg_0\|_{L^\infty_{x,v}},\quad
%	\sup_{0\leq s\leq T_*}\|w_jg^{n+1}(s)\|_{L^\infty_vL^2_x}\leq 2\|w_jg_0\|_{L^\infty_vL^2_x}.
%\end{align*}
Hence, \eqref{le.s.inftybd} and \eqref{le.s.L2bd} are proved for all $n\geq 0$.

Next, we prove that the constructed sequence $\{g^n\}_{n=0}^\infty$ is a Cauchy sequence in the $w_k$-weighted norm of $L^\infty_{x,v}$. By taking the difference and using similar arguments as  in \cite[Section 3]{Li}, one can obtain that
\begin{align*}
	&\dis \sup_{0\leq s\leq T_*}\left\|w_k g^{n+2}(s)-w_k g^{n+1}(s)\right\|_{L^\infty_{x,v}}\\
	&\leq CT_*(1+\|w_kg_0\|_{L^\infty_{x,v}})\sup_{0\leq s\leq T_*}\left\|w_k g^{n+1}-w_k g^{n}\right\|_{L^\infty_{x,v}}\notag\\
	&\leq \frac{C}{6C_k}\sup_{0\leq s\leq T_*}\left\|w_k g^{n+1}-w_k g^{n}\right\|_{L^\infty_{x,v}},
\end{align*}
for some $C>0$. Then we can choose $C_k$ to be large enough such that
\begin{align}\label{Cauchysoft}
	\dis \sup_{0\leq s\leq T_*}\left\|w_k g^{n+2}(s)-w_k g^{n+1}(s)\right\|_{L^\infty_{x,v}}&\leq \frac{1}{2}\sup_{0\leq s\leq T_*}\left\|w_k g^{n+1}-w_k g^{n}\right\|_{L^\infty_{x,v}}.
\end{align}
One may refer to Section 3 in \cite{Li} and Appendix in \cite{DHWY} for more details. Hence, we have proved that $\{g^n\}_{n=0}^\infty$ is a Cauchy sequence in the $w_k$-weighted norm of $L^\infty_{x,v}$. We take the limit  $n\to \infty$ to obtain a local mild solution $g\in \mathcal{X}_{T_*,k,j}$ satisfying \eqref{LE} in terms of the uniform estimates \eqref{le.s.inftybd} and \eqref{le.s.L2bd}. Recall $F^n\geq 0$ for any $n\geq 0$, so the limiting function $F=\mu+g$ is also nonnegative.
%Furthermore, \eqref{LE} holds from \eqref{le.s.inftybd} and \eqref{le.s.L2bd}.
For the uniqueness, let $g_1,g_2\in \mathcal{X}_{T_*,k,j}$ be two solutions to \eqref{rbe} satisfying \eqref{LE}.
%$\sup_{0\leq s\leq T_*}\|w_kg_i(s)\|_{L^\infty_{x,v}}\leq 2\|w_kg_0\|_{L^\infty_{x,v}}$ and $\sup_{0\leq s\leq T_*}\|w_jg_i(s)\|_{L^\infty_vL^2_x}\leq 2\|w_jg_0\|_{L^\infty_vL^2_x}$ for $i=1,2$.
Then, following the similar arguments for deducing \eqref{Cauchysoft}, it is direct to prove $g_1=g_2$. Thus, the case of soft potentials is proved.

For hard potentials $0\leq\ga\leq1$, we need to modify the proof above since the term $\nu(v)$ in \eqref{boundsoft} is unbounded when $v$ is large. As before, for some $T_*>0$ which will be determined later, we prove \eqref{le.s.inftybd} and \eqref{le.s.L2bd} for any $n\geq 0$ by induction. For the case that $n=0$, the above estimates hold naturally since $g^0\equiv 0$. We assume \eqref{le.s.inftybd} and \eqref{le.s.L2bd} for $n\geq 0$. Here, due to the difficulty caused by the nonlinear term, we make an additional assumption that $\|w_kg_0\|_{L^\infty_{x,v}}<\ep_1$. Recalling the definition of $\zeta$ that $\zeta^n(\tau,y,v)=\int_{\R^3}\int_{\S^2}B(v-u,\theta)\left[ \mu(u)+g^n (\tau,y,u) \right]d\omega du$, in order to overcome the unboundedness of $\nu(v)$, we start to deduce that $\zeta$ behaves like $\nu(v)$ up to a small positive time. We consider $0\leq \tau\leq T_*$. A direct calculation shows that
\begin{align}\label{zeta1}
	\zeta^n(\tau,y,v)&=\int_{\R^3}\int_{\S^2}B(v-u,\theta)\left[ \mu(u)+g^n (\tau,y,u) \right]d\omega du\notag\\&\geq \nu(v)-C\|w_k g^n(\tau)\|_{L^\infty_{x,v}}\int_{\R^3}\int_{\S^2}|v-u|^\ga(1+|u|)^{-k}d\om du.	
\end{align}
For $k\geq k_0>3$, it follows from \eqref{zeta1} that
\begin{align*}%\label{zeta2}
	\zeta^n(\tau,y,v)&\geq \nu(v)(1-C_k\sup_{0\leq s\leq T_*}\|w_kg^n(s)\|_{L^\infty_{x,v}}),	
\end{align*}
where $C_k$ depends only on $k$. Recalling $\sup_{0\leq s\leq T_*}\|w_k g^n(s)\|_{L^\infty_{x,v}}\leq 2\|w_kg_0\|_{L^\infty_{x,v}}$, we have
\begin{align*}%\label{zeta2}
	\zeta^n(\tau,y,v)&\geq \nu(v)(1-2C_k\|w_kg_0\|_{L^\infty_{x,v}}).	
\end{align*}
Letting $$\ep_1<\min\{1,\frac{1}{4C_k}\},$$ and defining $$T_1=\frac{1}{6C_k(1+\|w_k f_0\|_{L^\infty_{x,v}})},$$ then for $\|w_kg_0\|_{L^\infty_{x,v}}<\ep_1$ and $0\leq \tau\leq  T_*\leq T_1$,  one gets
\begin{align}\label{zeta}
	\zeta^n(\tau,y,v)&\geq \frac{1}{2}\nu(v).	
\end{align}
From \eqref{appro} and \eqref{zeta}, for $0\leq t\leq T_*\leq T_1$, it follows that
\begin{align}\label{boundhard}
	\dis \left|w_k(v) g^{n+1}(t,x,v)\right|\leq &\left|w_k(v) g_0(x-vt,v)\right|\notag\\
	&+\int_0^t e^{-\frac{1}{2}\nu(v)(t-s)}\left|w_k(v)(\CK g^n)(s,x-v(t-s),v)\right|ds\notag \\
	&+\int_0^te^{-\frac{1}{2}\nu(v)(t-s)} \left|w_k(v)Q_+(g^n,g^n)(s,x-v(t-s),v)\right|ds.
\end{align}
The first term on the right hand side above is directly controlled by
\begin{align}\label{local1}
	\left|w_k(v) g_0(x-vt,v)\right|\leq \|w_k g_0\|_{L^\infty_{x,v}}.
\end{align}
For the second term, for some $M>0$ to be determined later, we split it into two parts such that
\begin{align*}
	&\int_0^t e^{-\frac{1}{2}\nu(v)(t-s)}\left|w_k(v)(\CK g^n)(s,x-v(t-s),v)\right|ds\notag\\
	&=\int_0^t (\chi_{\{|v|\geq M\}}+\chi_{\{|v|< M\}}) e^{-\frac{1}{2}\nu(v)(t-s)}\left|w_k(v)(\CK g^n)(s,x-v(t-s),v)\right|ds.
\end{align*}
The part of large $v$ such that $|v|\geq M$ is bounded as
\begin{align*}
	&\int_0^t \chi_{\{|v|\geq M\}} e^{-\frac{1}{2}\nu(v)(t-s)}\left|w_k(v)(\CK g^n)(s,x-v(t-s),v)\right|ds\notag\\
	&\leq \left(\frac{C}{k^{\frac{\ga+3}{4}}}+\frac{C_k}{M^2}\right)	\int_0^t e^{-\frac{1}{2}\nu(v)(t-s)}\nu(v)ds	\sup_{0\leq s\leq T_*}\|w_kg^n(s)\|_{L^\infty_{x,v}}\notag\\
	&\leq \left(\frac{C}{k^{\frac{\ga+3}{4}}}+\frac{C_k}{M^2}\right)	\sup_{0\leq s\leq T_*}\|w_kg^n(s)\|_{L^\infty_{x,v}}.
\end{align*} 
For the other part over $|v|<M$, it holds from the boundedness of $v$ that
\begin{align*}
	&\int_0^t \chi_{\{|v|< M\}} e^{-\frac{1}{2}\nu(v)(t-s)}\left|w_k(v)(\CK g^n)(s,x-v(t-s),v)\right|ds\notag\\
	&\leq C_{k,M}\int_0^t e^{-\frac{1}{2}\nu(v)(t-s)}ds	\sup_{0\leq s\leq T_*}\|w_kg^n(s)\|_{L^\infty_{x,v}}\notag\\
	&\leq C_{k,M}T_*	\sup_{0\leq s\leq T_*}\|w_kg^n(s)\|_{L^\infty_{x,v}}.
\end{align*} 
The combination of the above two estimates gives
\begin{align}\label{local2}
	&\int_0^t e^{-\frac{1}{2}\nu(v)(t-s)}\left|w_k(v)(\CK g^n)(s,x-v(t-s),v)\right|ds\notag\\
	&\leq \left(\frac{C}{k^{\frac{\ga+3}{4}}}+\frac{C_k}{M^2}+C_{k,M}T_*\right)	\sup_{0\leq s\leq T_*}\|w_kg^n(s)\|_{L^\infty_{x,v}}.
\end{align}
For the last term on the right hand side of \eqref{boundhard}, it follows from \eqref{defQ} and \eqref{controlQ} that
\begin{align}\label{local3}
	&\int_0^te^{-\frac{1}{2}\nu(v)(t-s)} \left|w_k(v)Q_+(g^n,g^n)(s,x-v(t-s),v)\right|ds\notag\\
	&\leq C_k\int_0^te^{-\frac{1}{2}\nu(v)(t-s)}|v-u|^\ga\frac{w_k(v)}{w_k(v')w_k(u')}duds\sup_{0\leq s\leq T_*}\|w_k g^n(s)\|_{L^\infty_{x,v}}^2\notag\\
	&\leq C_k\int_0^te^{-\frac{1}{2}\nu(v)(t-s)}\nu(v)ds\sup_{0\leq s\leq T_*}\|w_k g^n(s)\|_{L^\infty_{x,v}}^2\notag\\
	&\leq C_k\sup_{0\leq s\leq T_*}\|w_k g^n(s)\|_{L^\infty_{x,v}}^2.
\end{align}
We collect \eqref{boundhard}, \eqref{local1}, \eqref{local2} and \eqref{local3} to obtain that
\begin{align}\label{boundseqhard}
	\dis \sup_{0\leq s\leq T_*}\|w_k g^{n+1}(s)\|_{L^\infty_{x,v}}\leq &\|w_k g_0\|_{L^\infty_{x,v}}+\left(\frac{C_1}{k^{\frac{\ga+3}{4}}}+\frac{C_k}{M^2}+C_{k,M}T_*\right)	\sup_{0\leq s\leq T_*}\|w_kg^n(s)\|_{L^\infty_{x,v}}\notag\\
	&+C_k\sup_{0\leq s\leq T_*}\|w_k g^n(s)\|_{L^\infty_{x,v}}^2,
\end{align}
where $C_1>0$ is a generic constant.
Similarly for obtaining \eqref{boundseqhard} from \eqref{local1}, \eqref{local2} and \eqref{local3}, one has \begin{align}\label{boundseqhardL2}
	\dis \|w_j g^{n+1}(t)\|_{L^\infty_vL^2_x}\leq &\|w_j g_0\|_{L^\infty_vL^2_x}+\left(\frac{C_1}{j^{\frac{\ga+3}{4}}}+\frac{C_j}{M^2}+C_{j,M}T_*\right)	\sup_{0\leq s\leq T_*}\|w_jg^n(s)\|_{L^\infty_vL^2_x}\notag\\
	&+C_j\sup_{0\leq s\leq T_*}\|w_k g^n(s)\|_{L^\infty_{x,v}}\|w_j g^n(s)\|_{L^\infty_vL^2_x}.
\end{align}
Based on the estimates above, we choose $k_0$ to be large enough such that $k\geq j\geq k_0>(4C_1)^\frac{2}{\ga+3}$ to obtain
$$
\frac{C_1}{k^{\frac{\ga+3}{4}}}< \frac{1}{4},\quad \frac{C_1}{j^{\frac{\ga+3}{4}}}< \frac{1}{4}.
$$
For such fixed $k$ and $j$, we further let $M$ be large enough such that
$$
\frac{C_k}{M^2}\leq \frac{1}{4},\quad \frac{C_j}{M^2}\leq \frac{1}{4}.
$$
Then, using the induction assumption \eqref{le.s.inftybd} and \eqref{le.s.L2bd}, it follows from \eqref{boundseqhard} and \eqref{boundseqhardL2} that 
\begin{align}\label{smalllinf}
	\dis \sup_{0\leq s\leq T_*}\|w_k g^{n+1}(s)\|_{L^\infty_{x,v}}\leq &\|w_k g_0\|_{L^\infty_{x,v}}+\left(\frac{1}{2}+{C}_{k,M}T_*+{C}_k\|w_k g_0\|_{L^\infty_{x,v}}\right)	\|w_k g_0\|_{L^\infty_{x,v}},
\end{align}
and \begin{align}\label{smalll2}
	\dis \sup_{0\leq s\leq T_*}\|w_j g^{n+1}(s)\|_{L^\infty_vL^2_x}\leq &\|w_j g_0\|_{L^\infty_vL^2_x}+\left(\frac{1}{2}+{C}_{j,M}T_*+{C}_j\|w_k g_0\|_{L^\infty_{x,v}}\right)	\|w_j g_0\|_{L^\infty_vL^2_x}.
\end{align}
%where $\overline{C}_k$ and $\overline{C}_j$ are constants that depend only on $k$ and $j$ respectively.
Furthermore, for fixed $k,j$ and $M$ as given above, we define \begin{align}\label{defT*}
	\dis T_*:=\min\{1,T_1,\frac{1}{4(1+{C}_{k,M}+{C}_{j,M})}\},
\end{align}
and let $\eps_1$ be small enough such that
$$
({C}_k+{C}_j)\|w_k g_0\|_{L^\infty_{x,v}}\leq ({C}_k+{C}_j)\eps_1\leq \frac{1}{4}.
$$
It then follows from \eqref{smalllinf} and \eqref{smalll2} that
$$
\sup_{0\leq s\leq T_*}\|w_kg^{n+1}(s)\|_{L^\infty_{x,v}}\leq 2\|w_kg_0\|_{L^\infty_{x,v}},
$$ 
and 
$$
\sup_{0\leq s\leq T_*}\|w_jg^{n+1}(s)\|_{L^\infty_vL^2_x}\leq 2\|w_jg_0\|_{L^\infty_vL^2_x},
$$ 
which give the same bounds as in \eqref{le.s.inftybd} and \eqref{le.s.L2bd} for $n+1$. Therefore,  \eqref{le.s.inftybd} and \eqref{le.s.L2bd} are satisfied for all $n\geq 0$ by induction.

Next, we prove that $\{g^n\}_{n=0}^\infty$  is a Cauchy sequence in the velocity weighted $L^\infty_{x,v}$ norm. Due to the growth of $\nu(v)$ in large velocities, we choose $k_0$ large enough such that  $k-2\ga\geq k_0-2\ga>0$ and consider the sequence $\{w_{k-2\ga}g^n\}_{n=0}^\infty$ instead of $\{w_{k}g^n\}_{n=0}^\infty$. Recalling \eqref{appro} and taking the difference, it is straightforward to get
 \begin{equation}\label{diffseq}
 	\dis |w_{k-2\ga}(v)(g^{n+2}-g^{n+1})(t,x,v)|\leq I_0(t,x,v)+I_1(t,x,v)+I_2(t,x,v)+I_3(t,x,v),
 \end{equation}
 where
 \begin{equation*}
 	I_0(t,x,v)=\left| e^{-{\int^t_0\zeta^{n+1}(\tau,x-v(t-\tau),v)d\tau}}-e^{-{\int^t_0\zeta^n(\tau,x-v(t-\tau),v)d\tau}}  \right|
	\cdot|w_{k-2\ga}(v) g_0(x-vt,v)|,
 \end{equation*}
 \begin{align*}
 	I_1(t,x,v)&=\int_0^t \left| e^{-{\int^t_s\zeta^{n+1}(\tau,x-v(t-\tau),v)d\tau}}-e^{-{\int^t_s\zeta^n(\tau,x-v(t-\tau),v)d\tau}}  \right|\notag\\
 	&\qquad\cdot\left(w_{k-2\ga}(v)|Kg^{n+1}(s,x-v(t-s),v)|\right.\notag\\
 	&\qquad\qquad\left.+w_{k-2\ga}(v)|Q_+(g^{n+1},g^{n+1})(s,x-v(t-s),v)|\right)ds,
 \end{align*}
  \begin{equation*}
 	I_2(t,x,v)=\int_0^t e^{-{\int^t_0\zeta^n(\tau,x-v(t-\tau),v)d\tau}} w_{k-2\ga}(v)
	\left|\left(Kg^{n+1}-Kg^{n}\right)(s,x-v(t-s),v)\right| ds,
 \end{equation*}
and
  \begin{multline*}
	I_3(t,x,v)=\int_0^t e^{-{\int^t_0\zeta^n(\tau,x-v(t-\tau),v)d\tau}} w_{k-2\ga}(v)\\
	\cdot\left|\left(Q_+(g^{n+1},g^{n+1})-Q_+(g^{n},g^{n})\right)(s,x-v(t-s),v)\right| ds.
\end{multline*}
Using the fact that $|e^{-a}-e^{-b}|\leq|a-b|$ for any $a,b\geq0$, we have
\begin{align}\label{estdiffe1}
&\left|e^{-{\int^t_s\zeta^{n+1}(\tau,x-v(t-\tau),v)d\tau}}-e^{-{\int^t_s\zeta^n(\tau,x-v(t-\tau),v)d\tau}}\right|\notag\\
&\leq \int^t_0\left| \int_{\R^3}\int_{\S^2}B(v-u,\theta)\left[g^{n+1} (\tau,y,u)-g^n (\tau,y,u) \right]d\omega du\right|d\tau\notag\\
&\leq C\sup_{0\leq s\leq T_*}\|w_{k-2\ga} g^{n+1}(s)-w_{k-2\ga} g^{n}(s)\|_{L^\infty_{x,v}}\int^t_0\int_{\R^3}|v-u|^\ga w_{2\ga-k}(u)dud\tau.
\end{align}
We let $k_0>6$ so as to deduce
$$
\int_{\R^3}|v-u|^\ga w_{2\ga-k}(u)du\leq C\nu(v).
$$
Then combining the above inequality with \eqref{estdiffe1}, one gets
\begin{align}\label{estdiffe}
	&\left|e^{-{\int^t_s\zeta^{n+1}(\tau,x-v(t-\tau),v)d\tau}}-e^{-{\int^t_s\zeta^n(\tau,x-v(t-\tau),v)d\tau}}\right|\notag\\
	&\leq CT_*\nu(v)\sup_{0\leq s\leq T_*}\|w_{k-2\ga} g^{n+1}(s)-w_{k-2\ga} g^{n}(s)\|_{L^\infty_{x,v}}.
\end{align}
As a special case that $s=0$ in the left hand side of \eqref{estdiffe}, it is direct to see that
 \begin{align}\label{I0}
	I_0(t,x,v)&\leq CT_*\sup_{0\leq s\leq T_*}\|w_{k-2\ga} g^{n+1}(s)-w_{k-2\ga} g^{n}(s)\|_{L^\infty_{x,v}}\notag\\
	&\qquad\qquad\cdot|w_{k-2\ga}(v)\nu(v) g_0(x-vt,v)|\notag\\
	&\leq CT_*\sup_{0\leq s\leq T_*}\|w_{k-2\ga} g^{n+1}(s)-w_{k-2\ga} g^{n}(s)\|_{L^\infty_{x,v}}\|w_kg_0\|_{L^\infty_{x,v}}.
\end{align}
For $I_1(t,x,v)$, from \eqref{estdiffe} and our choice that $T_*\leq 1$ in \eqref{defT*}, we have that
\begin{multline*}
I_1(t,x,v)\leq C\sup_{0\leq s\leq T_*}\|w_{k-2\ga} g^{n+1}(s)-w_{k-2\ga} g^{n}(s)\|_{L^\infty_{x,v}}\\
\cdot\int_0^t \left( \frac{w_k(v)}{\nu(v)}|Kg^{n+1}(s,x-v(t-s),v)|\right.\\
\left.+ \frac{w_k(v)}{\nu(v)}|Q_+(g^{n+1},g^{n+1})(s,x-v(t-s),v)|\right)ds.
%\label{I1}
\end{multline*}
Using similar arguments for obtaining \eqref{conseq}, \eqref{boundsoftn} and \eqref{boundsoft}, it further holds by \eqref{controlK}, \eqref{controlQ} and \eqref{le.s.inftybd} that
 \begin{multline}\label{I1}
	I_1(t,x,v)\leq T_*C_{k}\sup_{0\leq s\leq T_*}\|w_{k-2\ga} g^{n+1}(s)-w_{k-2\ga} g^{n}(s)\|_{L^\infty_{x,v}}\\
	\cdot\big(\|w_k g_0(s)\|_{L^\infty_{x,v}}+\|w_k g_0(s)\|_{L^\infty_{x,v}}^2\big).
\end{multline}
A similar argument to derive \eqref{boundseqhard} shows that
 \begin{equation}\label{I2}
	I_2(t,x,v)\leq T_*C_{k}\sup_{0\leq s\leq T_*}\|w_{k-2\ga} g^{n+1}(s)-w_{k-2\ga} g^{n}(s)\|_{L^\infty_{x,v}}.
\end{equation}
Using the fact that
$$Q_+(g^{n+1},g^{n+1})=Q_+(g^{n+1}-g^{n},g^{n})+Q_+(g^{n+1},g^{n+1}-g^{n}),
$$
one gets from \eqref{zeta} and the definition of $Q_+$ in \eqref{defQ} that
\begin{align*}
	&I_3(t,x,v)\\
	&\leq \int_0^t e^{-C\nu(v)(t-s)}w_{k-2\ga}(v)(|Q_+(g^{n+1}-g^{n},g^{n})|+|Q_+(g^{n+1},g^{n+1}-g^{n})|)ds\notag\\
	&\leq C(\sup_{0\leq s\leq T_*}\|w_k g^n(s)\|_{L^\infty_{x,v}}+\sup_{0\leq s\leq T_*}\|w_k g^{n+1}(s)\|_{L^\infty_{x,v}})\\
	&\quad\cdot\sup_{0\leq s\leq T_*}\|w_{k-2\ga} (g^{n+1}- g^{n})(s)\|_{L^\infty_{x,v}}\notag\\
	&\quad\cdot\int_0^t\int_{\R^3}\int_{\S^2}e^{-C\nu(v)(t-s)}|v-u|^\ga\frac{w_{k-2\ga}(v)}{w_{k-2\ga}(v')w_{k-2\ga}(u')}d\om duds.
\end{align*}
By our choice that $k_0>6$, we get $k-2\ga>3+\ga$. Then it follows from \eqref{controlQ2} and \eqref{le.s.inftybd} that
\begin{equation}\label{I3}
	I_3(t,x,v)\leq C\left(T_*C_{k}+\frac{C}{k^{\frac{\ga+3}{4}}}\right)\|w_k g_0(s)\|_{L^\infty_{x,v}}\sup_{0\leq s\leq T_*}\|w_{k-2\ga} g^{n+1}(s)-w_{k-2\ga} g^{n}(s)\|_{L^\infty_{x,v}}.
\end{equation}
Combining \eqref{diffseq}, \eqref{I0}, \eqref{I1}, \eqref{I2} and \eqref{I3}, and using the condition \eqref{thm.l.as} with $\ep\leq 1$, one has
\begin{equation*}
	\dis |w_{k-2\ga}(v)(g^{n+2}-g^{n+1})(t,x,v)|
	\leq \left(T_*C_{k}+\frac{C}{k^{\frac{\ga+3}{4}}}\right)\sup_{0\leq s\leq T_*}\|w_{k-2\ga} g^{n+1}(s)-w_{k-2\ga} g^{n}(s)\|_{L^\infty_{x,v}},
\end{equation*}
where $C_k$ depends only on $k$ and $C$ is independent of $k$ and $j$.
We may choose $C_1$ in \eqref{boundseqhard} so large that
\begin{equation*}
	\dis |w_{k-2\ga}(v)(g^{n+2}-g^{n+1})(t,x,v)|
	\leq \left(T_*C_{k}+\frac{C_1}{k^{\frac{\ga+3}{4}}}\right)\sup_{0\leq s\leq T_*}\|w_{k-2\ga} g^{n+1}(s)-w_{k-2\ga}g^{n}(s)\|_{L^\infty_{x,v}}.
\end{equation*}
Then by our choice that $k_0>(4C_1)^\frac{2}{\ga+3}$, we can deduce from the definition of $T_*$ in \eqref{defT*} that for $0\leq t\leq T_*$,
\begin{equation}\label{Cauchyhard}
	\dis \|w_{k-2\ga}g^{n+2}(t)-w_{k-2\ga}g^{n+1}(t)\|_{L^\infty_{x,v}}\leq\frac{1}{2}\sup_{0\leq s\leq T_*}\|w_{k-2\ga} g^{n+1}(s)-w_{k-2\ga} g^{n}(s)\|_{L^\infty_{x,v}}.
\end{equation}
 Hence, $\{w_{k-2\ga}g^n\}^\infty_{n=0}$ is a Cauchy sequence in $L^\infty_{T_*,x,v}$. We take the limit to obtain a local mild solution.
Recalling $F^n\geq 0$ for any $n\geq 0$, so the limiting function $F=\mu+g$ is also nonnegative. Also \eqref{LE} follows by letting $n\rightarrow\infty$ in \eqref{le.s.inftybd} and \eqref{le.s.L2bd}. For the uniqueness, let $g_1,g_2\in \mathcal{X}_{T_*,k,j}$ be two solutions to \eqref{rbe} satisfying \eqref{thm.l.as}.
%, $\sup_{0\leq s\leq T_*}\|w_kg_i(s)\|_{L^\infty_{x,v}}\leq 2\|w_kg_0\|_{L^\infty_{x,v}}$ and $\sup_{0\leq s\leq T_*}\|w_jg_i(s)\|_{L^\infty_vL^2_x}\leq 2\|w_jg_0\|_{L^\infty_vL^2_x}$ for $i=1,2$.
Then, following the similar arguments for deriving \eqref{Cauchyhard}, it is straightforward to prove $g_1=g_2$.  We then conclude the proof of Theorem \ref{local}.
 \end{proof}

\section{Hard potential case}\label{sec4}

Motivated by \cite{DL-arma,DL-cmaa,DL-hard} for construction of polynomial tail solutions, we first resolve the problem \eqref{rbe} into a coupling system of $g_1=g_1(t,x,v)$ and $g_2=g_2(t,x,v)$ where $g_1$ and $g_2$ satisfy
\begin{align}
	\pa_t g_1+v\cdot\nabla_x g_1 +\nu g_1&=\CK_s g_1+Q(g_1+\sqrt{\mu}g_2,g_1+\sqrt{\mu}g_2), \label{g1} \\
	\pa_t g_2 +v\cdot\na_x g_2 +Lg_2 &=\CK_b g_1,  \label{g2}
\end{align}
with
\begin{equation}
	\label{initial}
	g_1(0,x,v)=g_{10}(x,v)=F_0(x,v)-\mu(v),\quad g_2(0,x,v)=g_{20}(x,v)=0.
\end{equation}
The linear operators $\CK_s$ and $\CK_b$ above are respectively defined by
\begin{align}\label{defKs}
	\CK_sg_1(t,x,v):=\chi_{\{|v|\geq M\}} \CK g_1(t,x,v),
\end{align}
and
\begin{align}\label{defKb}
	\CK_bg_1(t,x,v):=\chi_{\{|v|< M\}}\mu^{-1/2}(v) \CK g_1(t,x,v),
\end{align}
where for any set $E$, $\chi_{\{v\in E\}}=1$ if $v\in E$ and $\chi_{\{v\in E\}}=0$ otherwise. The constant $M$ which depends on $k$ will be chosen later in \eqref{chooseM}. By setting $g=g_1+\sqrt{\mu}g_2$ with $g_0=g_{10}+\sqrt{\mu}g_{20}$, it is straightforward to see  that $g$ is the solution to \eqref{rbe}. We now focus on the system \eqref{g1}, \eqref{g2} and \eqref{initial}.

\subsection{Estimates on $g_1$}
We first consider estimates on $g_1$ from the equation \eqref{g1}. These are related to the velocity weighted norms in $L^\infty_{x,v}$, $L^\infty_vL^2_x$ and $L^\infty_vL^1_x$.

\begin{lemma}\label{leg1hard}
Let $0\leq\gamma\leq 1$, then there is $k_0>0$ large enough such that for any $k\geq k_0$, there is $M>0$ for the decomposition $\CK=\CK_s+\sqrt{\mu}\CK_b$  such that the following estimates hold:
	\begin{align}
	\|w_kg_1(t)\|_{L^\infty_{x,v}}&\leq C(1+t)^{-\frac{3}{2}}\|w_kg_{10}\|_{L^\infty_{x,v}}\notag\\
	&\quad+C_k(1+t)^{-\frac{3}{2}}\sup_{0\leq s\leq t}\|(1+s)^\frac{3}{4}w_kg(s)\|^2_{L^\infty_{x,v}},\label{g1infty}
\end{align}
\begin{align}
\|w_{k}g_1(t)\|_{L^\infty_vL^2_x}&\leq C(1+t)^{-\frac{3}{2}}\|w_{k}g_{10}\|_{L^\infty_vL^2_x}\notag\\
	&\quad+C_k(1+t)^{-\frac{3}{2}}\left\{\sup_{0\leq s\leq t}\|(1+s)^\frac{3}{4}w_kg(s)\|^2_{L^\infty_{x,v}}\right.\notag\\
	&\qquad\qquad\qquad\qquad\qquad\left.+\sup_{0\leq s\leq t}\|(1+s)^\frac{3}{4}w_{k}g(s)\|^2_{L^\infty_vL^2_x}\right\},\label{g1L2}
\end{align}
and
\begin{align}
\|w_{k}g_1(t)\|_{L^\infty_vL^1_x}&\leq C(1+t)^{-\frac{3}{2}}\|w_{k}g_{10}\|_{L^\infty_vL^1_x}\notag\\
&\quad+C_k(1+t)^{-\frac{3}{2}}\sup_{0\leq s\leq t}\|(1+s)^\frac{3}{4}w_{k}g(s)\|^2_{L^\infty_vL^2_x}\label{g1L1},
\end{align}
where $C_k$ is a constant which may depend on $k$.
\end{lemma}
\begin{proof}
	We integrate along the characteristic in \eqref{g1} to get
	\begin{align}\label{mildg1}
		\dis g_1(t,x,v)=&e^{-\nu(v)t}g_{10}(x-vt,v)+\int_0^t e^{-\nu(v)(t-s)}\CK_s g_1(s,x-v(t-s),v)ds \notag\\
		&+\int_0^t e^{-\nu(v)(t-s)}Q(g,g)(s,x-v(t-s),v)ds.
	\end{align}
We now start to prove the first estimate \eqref{g1infty}. We directly obtain from \eqref{mildg1} that
\begin{equation}\label{g11}
	\left|w_k(v)g_1(t,v)\right|\leq Ce^{-\la t}\|w_kg_{10}\|_{L^\infty_{x,v}}+I_4(t,x,v)+I_5(t,x,v),
\end{equation}
with
	\begin{align*}%\label{g11}
	I_4(t,x,v):=&\int_0^t e^{-\nu(v)(t-s)}\left|w_k(v)\CK_s g_1(s,x-v(t-s),v)\right|ds,\\
	I_5(t,x,v):=&\int_0^t e^{-\nu(v)(t-s)}\left|w_k(v)Q(g,g)(s,x-v(t-s),v)\right|ds,
	\end{align*}
where $\la>0$ is a constant. We choose $k_0>3+\ga$ such that Lemma \ref{leQ} and Lemma \ref{leKs} hold for $k\geq k_0$. Taking the $L^\infty_x$ norm and using the definition of $\CK_s$ in \eqref{defKs}, we have
\begin{align*}
	\|I_4(t,v)\|_{L^\infty_x}\leq &\int_0^t e^{-\nu(v)(t-s)}\|\chi_{\{|v|>M\}}w_k(v)\CK g_1(t,v)\|_{L^\infty_x}ds.
\end{align*}
It follows from \eqref{smallCK} that
\begin{equation*}
	\|I_4(t,v)\|_{L^\infty_x}\leq C	\sup_{0\leq s\leq t}\|(1+s)^\frac{3}{2}w_k g_1(s)\|_{L^\infty_{x,v}}\int_0^t e^{-\nu(v)(t-s)}\nu(v)(1+s)^{-\frac{3}{2}}\left(\frac{C}{k^{\frac{\ga+3}{2}}}+\frac{C_k}{M^2}\right)ds,
\end{equation*}
which combining with \eqref{decay} yields
\begin{align}\label{C2}
	\|I_4(t,v)\|_{L^\infty_x}\leq &C_2\left(\frac{1}{k^{\frac{\ga+3}{2}}}+\frac{C_k}{M^2}\right)	(1+t)^{-\frac{3}{2}}\sup_{0\leq s\leq t}\|(1+s)^\frac{3}{2}w_k g_1(s)\|_{L^\infty_{x,v}},
\end{align}
for some constant $C_2$. We choose $k_0$ large enough, for instance, $k_0>(4C_2)^\frac{2}{\ga+3}$,  so that for any $k\geq k_0$,
\begin{align}\label{largek}
\frac{C_2}{k^{\frac{\ga+3}{2}}}\leq \frac{1}{4}.
\end{align}
Then we let
\begin{align}\label{chooseM}
M=2\sqrt{C_2C_k}
\end{align}
to get from \eqref{C2} and \eqref{largek} that
\begin{align}\label{I4}
	\|I_4(t,v)\|_{L^\infty_x}\leq\frac{1}{2}(1+t)^{-\frac{3}{2}}\sup_{0\leq s\leq t}\|(1+s)^\frac{3}{2}w_k g_1(s)\|_{L^\infty_{x,v}}.
\end{align}
For $I_5$, by the definition of $Q$ in \eqref{defQ}, one gets
\begin{align}\label{I51}
	|I_5(t,x,v)|&\leq \int_0^t e^{-\nu(v)(t-s)}\left|w_k(v)Q(g,g)(s,x-v(t-s),v)\right|ds\notag\\
	&\leq C\sup_{0\leq s\leq t}\|(1+s)^\frac{3}{4}w_kg(s)\|^2_{L^\infty_{x,v}}\int_0^t e^{-\nu(v)(t-s)}(1+s)^{-\frac{3}{2}}\notag\\
	&\qquad\qquad\times\int_{\R^3}\int_{\S^2}|v-u|^\ga\left( \frac{w_k(v)}{w_k(u')w_k(v')}+\frac{1}{w_k(u)}\right) d\omega duds.
\end{align}
Noticing the fact that $\int_{\R^3}\frac{|v-u|^\ga}{w_k(u)}d\omega du\leq C_k\nu(v)$ for $k\geq k_0>3+\ga$, we deduce by \eqref{controlQ} and \eqref{decay} that
\begin{align}\label{I5}
	|I_5(t,x,v)|&\leq C_k\sup_{0\leq s\leq t}\|(1+s)^\frac{3}{4}w_kg(s)\|^2_{L^\infty_{x,v}}\int_0^t e^{-\nu(v)(t-s)}\nu(v)(1+s)^{-\frac{3}{2}}ds\notag\\
	&\leq C_k(1+t)^{-\frac{3}{2}}\sup_{0\leq s\leq t}\|(1+s)^\frac{3}{4}w_kg(s)\|^2_{L^\infty_{x,v}}.
\end{align}
Collecting \eqref{g11}, \eqref{I4} and \eqref{I5}, we obtain
\begin{align}\label{g1infty1}
	\dis \|w_kg_1(t)\|_{L^\infty_{x,v}}\leq&Ce^{-\la t}\|w_kg_{10}\|_{L^\infty_{x,v}}+\frac{1}{2}	(1+t)^{-\frac{3}{2}}\sup_{0\leq s\leq t}\|(1+s)^\frac{3}{2}w_k g_1(s)\|_{L^\infty_{x,v}} \notag\\
	&+C_k(1+t)^{-\frac{3}{2}}\sup_{0\leq s\leq t}\|(1+s)^\frac{3}{4}w_kg(s)\|^2_{L^\infty_{x,v}},
\end{align}
which yields
\begin{align}\label{g1est1}
	\dis \sup_{0\leq s\leq t}\|(1+s)^\frac{3}{2}w_k g_1(s)\|_{L^\infty_{x,v}} \leq&C\|w_kg_{10}\|_{L^\infty_{x,v}}+C_k\sup_{0\leq s\leq t}\|(1+s)^\frac{3}{4}w_kg(s)\|^2_{L^\infty_{x,v}}.	
\end{align}
Then \eqref{g1infty} holds from the above inequality.

We turn to \eqref{g1L2}. It follows from similar arguments as in \eqref{g11} and \eqref{I4} that
 \begin{multline}\label{wk0g1L2}
 	\dis \|w_{k}g_1(t)\|_{L^\infty_vL^2_x}\leq Ce^{-\la t}\|w_{k}g_{10}\|_{L^\infty_vL^2_x}\\
 	+\frac{1}{2}	(1+t)^{-\frac{3}{2}}\sup_{0\leq s\leq t}\|(1+s)^\frac{3}{2}w_{k} g_1(s)\|_{L^\infty_vL^2_x}+\|I_5(t)\|_{L^\infty_vL^2_x},
 \end{multline}
where $I_5$ is given in \eqref{g11}. Applying similar arguments as in \eqref{I4} and \eqref{I51}, we have
\begin{align*}
	\|I_5(t,v)\|_{L^2_x}&\leq C\int_0^t e^{-\nu(v)(t-s)}(1+s)^{-\frac{3}{2}}\int_{\R^3}\int_{\S^2}|v-u|^\ga\left(\frac{w_{k}(v)}{w_{k}(v')w_{k}(u')}+\frac{1}{w_{k}(u)}\right)\notag\\
	&\qquad\qquad\times \|(1+s)^\frac{3}{4}w_{k}g(s)\|_{L^\infty_vL^2_x}\|(1+s)^\frac{3}{4}w_{k}g(s)\|_{L^\infty_{x,v}}d\omega duds.
\end{align*}	
	Again by \eqref{controlQ}, Cauchy-Schwarz inequality and \eqref{decay}, one gets	
\begin{align}\label{I5L2}	
		&\|I_5(t,v)\|_{L^2_x}\notag\\
		&\leq C_k\int_0^t e^{-\nu(v)(t-s)}(1+s)^{-\frac{3}{2}}\nu(v)\notag\\
		&\qquad\qquad\times\left\{\sup_{0\leq s\leq t}\|(1+s)^\frac{3}{4}w_{k}g(s)\|^2_{L^\infty_{x,v}}+\sup_{0\leq s\leq t}\|(1+s)^\frac{3}{4}w_{k}g(s)\|^2_{L^\infty_vL^2_x}\right\}\notag\\
		&\leq C_k(1+t)^{-\frac{3}{2}}\left\{\sup_{0\leq s\leq t}\|(1+s)^\frac{3}{4}w_{k}g(s)\|^2_{L^\infty_{x,v}}+\sup_{0\leq s\leq t}\|(1+s)^\frac{3}{4}w_{k}g(s)\|^2_{L^\infty_vL^2_x}\right\}.
\end{align}
Combining \eqref{wk0g1L2} and \eqref{I5L2}, similar arguments in \eqref{g1infty1} and \eqref{g1est1} show that
\begin{multline*}
	\dis \sup_{0\leq s\leq t}\|(1+s)^\frac{3}{2} w_{k} g_1(s)\|_{L^\infty_vL^2_x} \leq C\|w_{k}g_{10}\|_{L^\infty_vL^2_x}\\
	+C_k\left\{\sup_{0\leq s\leq t}\|(1+s)^\frac{3}{4}w_{k}g(s)\|^2_{L^\infty_{x,v}}+\sup_{0\leq s\leq t}\|(1+s)^\frac{3}{4}w_{k}g(s)\|^2_{L^\infty_vL^2_x}\right\},
\end{multline*}
which yields \eqref{g1L2}.

As for \eqref{g1L1}, it follows from similar arguments in \eqref{g11}, \eqref{I4} and \eqref{wk0g1L2} that
 \begin{multline}\label{wk0g1L1}
	\dis \|w_{k}g_1(t)\|_{L^\infty_vL^1_x}\leq Ce^{-\la t}\|w_{k}g_{10}\|_{L^\infty_vL^1_x}\\
	+\frac{1}{2}	(1+t)^{-\frac{3}{2}}\sup_{0\leq s\leq t}\|(1+s)^\frac{3}{2}w_{k} g_1(s)\|_{L^\infty_vL^1_x}+\|I_5(t)\|_{L^\infty_vL^1_x}.
\end{multline}
We use \eqref{controlQ} and \eqref{decay} just as in \eqref{I4} and \eqref{I51} to obtain
\begin{align}\label{I5L1}	
	\|I_5(t,v)\|_{L^1_x}&\leq C\int_0^t e^{-\nu(v)(t-s)}(1+s)^{-\frac{3}{2}}\int_{\R^3}\int_{\S^2}|v-u|^\ga\left(\frac{w_{k}(v)}{w_{k}(v')w_{k}(u')}+\frac{1}{w_{k}(u)}\right)\notag\\
	&\qquad\qquad\qquad\quad\qquad\times \|(1+s)^\frac{3}{4}w_{k}g(s)\|^2_{L^\infty_vL^2_x}d\omega duds\notag\\
	&\leq C_k\int_0^t e^{-\nu(v)(t-s)}(1+s)^{-\frac{3}{2}}\nu(v)\sup_{0\leq s\leq t}\|(1+s)^\frac{3}{4}w_{k}g(s)\|^2_{L^\infty_vL^2_x}\notag\\
	&\leq C_k(1+t)^{-\frac{3}{2}}\sup_{0\leq s\leq t}\|(1+s)^\frac{3}{4}w_{k}g(s)\|^2_{L^\infty_vL^2_x}.
\end{align}
The first inequality in the above estimate holds by the fact that for any $s\geq0$, $x,\xi,\eta\in \R^3$,
 \begin{align*}%\label{CS}
 \int_{\R^3}(1+s)^\frac{3}{2} |(w_{k}g)(s,x,\xi)||(w_{k}g)(s,x,\eta)|dx\leq \|(1+s)^\frac{3}{4}w_{k}g(s)\|^2_{L^\infty_vL^2_x}.
 \end{align*}
  Then it holds from \eqref{wk0g1L1} and \eqref{I5L1} that
 \begin{align*}
 	\dis \sup_{0\leq s\leq t}\|(1+s)^\frac{3}{2}&w_{k} g_1(s)\|_{L^\infty_vL^1_x} \leq C\|w_{k}g_{10}\|_{L^\infty_vL^1_x}+C_k\sup_{0\leq s\leq t}\|(1+s)^\frac{3}{4}w_{k}g(s)\|^2_{L^\infty_vL^1_x},	
 \end{align*}
which proves \eqref{g1L1} and concludes the proof of Lemma \ref{leg1hard}.
\end{proof}

\subsection{Time decay in the symmetric case}
After obtaining uniform estimates on $g_1$ from the equation \eqref{g1}, we now focus on the other equation \eqref{g2} for estimates on $g_2$. Recall \eqref{def.L} for the self-adjoint operator $L$. We first consider the corresponding linear homogeneous problem
%which corresponds to the equation
\begin{align}\label{linearf}
	\pa_t f +v\cdot\na_x f +Lf &=0,\qquad f(0,x,v)=f_{0}(x,v).
\end{align}
The following proposition gives the large time behavior of solutions to  \eqref{linearf} for initial data in $L^2_{x,v}\cap L^2_vL^1_x$. Since the proof is quite standard, we omit it for brevity; see \cite{DSa} and references therein, for instance.

\begin{proposition}\label{propL2}
	Let $0\leq \ga\leq 1$. Let $f$ be the solution to the problem \eqref{linearf}, then it holds
	\begin{align}\label{linearL2}
		\|f(t)\|_{L^2_{x,v}}\leq C(1+t)^{-\frac{3}{4}}(\|f_0\|_{L^2_{x,v}}+\|f_0\|_{L^2_vL^1_x}),
	\end{align}
for any $t\geq0$.
\end{proposition}
With the help of the above proposition, we can prove the $L^\infty_{x,v}$ and $L^\infty_vL^2_x$ decay properties stated as follows.

\begin{lemma}\label{lem4.2}
	Let $0\leq \ga\leq 1$ and $k>3+\ga$. Let $f$ be the solution to the problem \eqref{linearf}, then it holds
	\begin{align}
		&\|w_{k}f(t)\|_{L^\infty_{x,v}}\leq C_k(1+t)^{-\frac{3}{4}}(\|w_{k}f_0\|_{L^\infty_{x,v}}+\|f_0\|_{L^2_{x,v}}+\|f_0\|_{L^2_vL^1_x}), \label{linearLinf}\\
				&\|w_{k}f(t)\|_{L^\infty_vL^2_x}\leq C_k(1+t)^{-\frac{3}{4}}(\|w_{k}f_0\|_{L^\infty_vL^2_x}+\|f_0\|_{L^2_vL^1_x}), \label{linearLinfL2}
	\end{align}
for any $t\geq0$.
\end{lemma}
\begin{proof}
	Recall the definition of the function $k=k(v,\eta)$ in Lemma \ref{propK}. Define
	\begin{align}\label{Defkw}
		k_{w_k}(v,\eta)=k(v,\eta)\frac{w_k(v)}{w_k(\eta)}
	\end{align}
and
	\begin{align*}
		K_{w_k}f(v)=\int_{\R^3}k_{w_k}(v,\eta)f(\eta)d\eta.
	\end{align*}
	We can rewrite the equation \eqref{linearf} as
	\begin{align*}
		\pa_t w_{k}f +v\cdot\na_x w_{k}f +\nu w_{k}f-K_{w_k}w_{k}f=0.
	\end{align*}
Then
\begin{align*}
	w_{k}(v)f(t,x,v)
%	&=e^{-\nu(v)t}w_{k}(v)f_0(x-vt,v)\\
%	&\quad+\int_0^t e^{-\nu(v)(t-s)}K_{w_k}(w_{k}f)(s,x-v(t-s),v)ds\notag\\
	&=e^{-\nu(v)t}w_{k}(v)f_0(x-vt,v)\notag\\
	&\quad+\int_0^t e^{-\nu(v)(t-s)}\int_{\R^3}k_{w_k}(v,\eta)(w_{k}f)(s,x-v(t-s),\eta)d\eta ds.
\end{align*}
We further have
	\begin{align}\label{iteration}
	&w_{k}(v)f(t,x,v)\notag\\
	&=e^{-\nu(v)t}w_{k}(v)f_0(x-vt,v)\notag\\
	&\quad+\int_0^t e^{-\nu(v)(t-s)}\int_{\R^3}k_{w_k}(v,\eta)e^{-\nu(\eta)s}w_{k}(\eta)f_0(x-v(t-s),\eta)d\eta ds\notag\\
	&\quad+\int_0^t e^{-\nu(v)(t-s)}\int_{\R^3}k_{w_k}(v,\eta)\int_0^s e^{-\nu(\eta)(s-s_1)}\notag\\
	&\qquad\quad\times\int_{\R^3}k_{w_k}(\eta,\xi)w_{k}(\xi)f(s_1,x_1-\eta(s-s_1),\xi)d\xi ds_1d\eta  ds,
\end{align}
where $x_1=x-v(t-s)$.

For the $L^\infty_{x,v}$ estimate \eqref{linearLinf}, by the fact that $\nu(v)\geq \nu_0$ for some constant $\nu_0>0$, the first term on the right hand side above is directly bounded by $e^{-\nu_0 t}\|w_{k}f_0\|_{L^\infty_{x,v}}$. For the second term, it holds from \eqref{Prok} that
\begin{align*}
	&\int_0^t e^{-\nu(v)(t-s)}\int_{\R^3}k_{w_k}(v,\eta)e^{-\nu(\eta)s}w_{k}(\eta)f_0(x-v(t-s),\eta)d\eta ds\notag\\
	&\leq Cte^{-\nu_0 t}\|w_{k}f_0\|_{L^\infty_{x,v}}\int_{\R^3}k_{w_k}(v,\eta)d\eta\notag\\
	&\leq C_ke^{-\frac{\nu_0}{2} t}\|w_{k}f_0\|_{L^\infty_{x,v}}.
\end{align*}
Then we have
\begin{align}\label{wf1}
	|w_{k}(v)f(t,x,v)|\leq C_ke^{-\la_0 t}\|w_{k}f_0\|_{L^\infty_{x,v}}+I_6(t,x,v),
\end{align}
where $\la_0>0$ is a constant and
\begin{multline}\label{defI6}
	I_6(t,x,v)=\int_0^t e^{-\nu(v)(t-s)}\int_{\R^3}\int_{\R^3}\left|k_{w_k}(v,\eta)k_{w_k}(\eta,\xi)\right|\\
	\times\int_0^s e^{-\nu(\eta)(s-s_1)}\left|w_{k}(\xi)f(s_1,x_1-\eta(s-s_1),\xi)\right|d\xi ds_1d\eta  ds.
\end{multline}
 We split $I_6(t,x,v)$ into four cases as \cite{DHWY,GuoY,Guo}.

\medskip
\noindent{\it Case 1. } $|v|\geq N$. We directly have from \eqref{defI6} and \eqref{tdecay1} that
\begin{align}\label{I6c1}
	I_6(t,x,v) &\leq C\sup_{0\leq s\leq t}\|(1+s)^\frac{3}{4}w_{k}f(s)\|_{L^\infty_{x,v}}\notag\\
	&\qquad\times\int_0^t e^{-\nu_0(t-s)}\int_{\R^3}\int_{\R^3}\left|k_{w_k}(v,\eta)k_{w_k}(\eta,\xi)\right|\notag\\
	&\qquad\qquad\times\int_0^s e^{-\nu_0(s-s_1)}(1+s_1)^{-\frac{3}{4}}d\xi ds_1d\eta  ds\notag\\
	&\leq C(1+t)^{-\frac{3}{4}}\sup_{0\leq s\leq t}\|(1+s)^\frac{3}{4}w_{k}f(s)\|_{L^\infty_{x,v}}\notag\\
	&\qquad\qquad\times\int_{\R^3}\int_{\R^3}\left|k_{w_k}(v,\eta)k_{w_k}(\eta,\xi)\right|d\xi d\eta.
\end{align}
 It follows from \eqref{Prok} that
 \begin{align}\label{largevc1}
 	\int_{\R^3}\int_{\R^3}\left|k_{w_k}(v,\eta)k_{w_k}(\eta,\xi)\right|d\xi d\eta\leq \frac{C_k}{1+|v|}\leq \frac{C_k}{N}.
 \end{align}
Then combining \eqref{I6c1} and \eqref{largevc1}, we get
 \begin{align}\label{HC1}
 	I_6(t,x,v)&\leq \frac{C_k}{N}(1+t)^{-\frac{3}{4}}\sup_{0\leq s\leq t}\|(1+s)^\frac{3}{4}w_{k}f(s)\|_{L^\infty_{x,v}}.
 \end{align}

\medskip
\noindent{\it Case 2. } $|v|\leq N$, $|\eta|\geq 2N$ or $|\eta|\leq2N$, $|\xi|\geq3N$. In this case, we have either $|v-\eta|\geq N$ or $|\eta-\xi|\geq N$. Similar arguments as above show that
\begin{align}\label{HC2}
	I_6(t,x,v)&\leq C e^{-\frac{N^2}{20}}\sup_{0\leq s\leq t}\|(1+s)^\frac{3}{4}w_{k}f(s)\|_{L^\infty_{x,v}}\int_0^t e^{-\nu_0(t-s)}\notag\\
	&\qquad \cdot\int_{\R^3}\int_{\R^3}\left|k_{w_k}(v,\eta)e^{\frac{|v-\eta|^2}{20}}k_{w_k}(\eta,\xi)e^{\frac{|\eta-\xi|^2}{20}}\right|\notag\\
	&\qquad\qquad\cdot\int_0^s e^{-\nu_0(s-s_1)}(1+s_1)^{-\frac{3}{4}}d\xi ds_1d\eta  ds\notag\\
	&\leq \frac{C_k}{N}(1+t)^{-\frac{3}{4}}\sup_{0\leq s\leq t}\|(1+s)^\frac{3}{4}w_{k}f(s)\|_{L^\infty_{x,v}}.
\end{align}
\medskip
\noindent{\it Case 3. } $s-s_1\leq\la\ll1$. It follows from the fact $s-\la\leq s_1\leq s$ that
\begin{align}\label{HC3}
	I_6(t,x,v)&\leq C\sup_{0\leq s\leq t}\|(1+s)^\frac{3}{4}w_{k}f(s)\|_{L^\infty_{x,v}}\notag\\
	&\qquad\times\int_0^t e^{-\nu_0(t-s)}\int_{\R^3}\int_{\R^3}\left|k_{w_k}(v,\eta)k_{w_k}(\eta,\xi)\right|\notag\\
	&\qquad\qquad\cdot\int_{s-\la}^s e^{-\nu_0(s-s_1)}(1+s_1)^{-\frac{3}{4}}d\xi ds_1d\eta  ds\notag\\
	&\leq C_k\la(1+t)^{-\frac{3}{4}}\sup_{0\leq s\leq t}\|(1+s)^\frac{3}{4}w_{k}f(s)\|_{L^\infty_{x,v}}.
\end{align}

\medskip
\noindent{\it Case 4. } $|v|\leq N$, $|\eta|\leq2N$, $|\xi|\leq3N$, $s-s_1\geq\la$.
This case needs to be treated more carefully. By the property of $k$ \eqref{Prok}, we can approximate $k_{w_k}$ by a smooth function $k_N$ with compact support such that
\begin{align}\label{app}
	\sup_{|v|\leq3N}\int_{|\eta|\leq3N}\left| k_{w_k}(v,\eta)-k_N(v,\eta) \right|d\eta \leq \frac{C_k}{N}.
\end{align}
We can split $I_6$
\begin{align*}
	\dis I_6(t,x,v)&\leq C\sup_{0\leq s\leq t}\|(1+s)^\frac{3}{4}w_{k}f(s)\|_{L^\infty_{x,v}}\int^t_0 e^{-\nu_0(t-s)}\notag\\
	&\quad\times \int_{\R^3}\int_{\R^3}\left|k_{w_k}(v,\eta)-k_{N}(v,\eta)\right|\left|k_{w_k}(\eta,\xi)\right|\int^{s-\la}_0e^{-\nu_0(s-s_1)}(1+s_1)^{-\frac{3}{4}}ds_1d\eta d\xi ds \notag\\
	&+C\sup_{0\leq s\leq t}\|(1+s)^\frac{3}{4}w_{k}f(s)\|_{L^\infty_{x,v}}\int^t_0 e^{-\nu_0(t-s)}\notag\\
	&\quad\times \int_{\R^3}\int_{\R^3}\left|k_{N}(v,\eta)\right|\left|k_{w_k}(\eta,\xi)-k_{N}(\eta,\xi)\right|\int^{s-\la}_0e^{-\nu_0(s-s_1)}(1+s_1)^{-\frac{3}{4}}ds_1d\eta d\xi ds \notag\\
	&+\int^t_0 e^{-\nu(v)(t-s)}\iint_{|\eta|\leq2N,|\xi|\leq3N}\left|k_{N}(v,\eta)k_{N}(\eta,\xi)\right|\notag\\
	&\quad\times \int^{s-\la}_0e^{-\nu(\eta)(s-s_1)}\left|(w_k f)(s_1,x_1-\eta(s-s_1),\xi)\right|ds_1d\eta d\xi  ds.
\end{align*}
Combining the above inequality and \eqref{app}, we get
\begin{align}\label{HC4}
	\dis I_6(t,x,v)&\leq \frac{C_k}{N}(1+t)^{-\frac{3}{4}}\sup_{0\leq s\leq t}\|(1+s)^\frac{3}{4}w_{k}f(s)\|_{L^\infty_{x,v}}+I_{61}(t,x,v),
\end{align}
where
\begin{align*}%\label{defI61}
	\dis I_{61}(t,x,v)&=C_{k,N}\int^t_0 e^{-\nu_0(t-s)}\int^{s-\la}_0e^{-\nu_0(s-s_1)}\notag\\
	&\qquad\qquad\times\iint_{|\eta|\leq2N,|\xi|\leq3N}\left|f(s_1,x_1-\eta(s-s_1),\xi)\right|d\eta d\xi  ds_1ds.
\end{align*}
Using Cauchy-Schwarz's inequality and change of variable $y=x_1-\eta(s-s_1)$, we have
 \begin{align}\label{I611}
 	\dis I_{61}(t,x,v)& \leq C_{k,N}\int^t_0e^{-\nu_0(t-s)}\int^{s-\la}_0e^{-\nu_0(s-s_1)}\notag\\
 	&\qquad\quad\times\left(\iint_{|\eta|\leq2N,|\xi|\leq3N}\left|f(s_1,x_1-\eta(s-s_1),\xi)\right|^2d\eta d\xi\right)^\frac{1}{2}  ds_1ds\notag\\
 	&\leq C_{k,N,\la}\int^t_0e^{-\nu_0(t-s)}\int^{s-\la}_0e^{-\nu_0(s-s_1)}\left(\int_{\R^3}\int_{\R^3}\left|f(s_1,y,\xi)\right|^2dy d\xi\right)^\frac{1}{2}  ds_1ds\notag\\
 	&\leq C_{k,N,\la}\int^t_0e^{-\nu_0(t-s)}\int^{s-\la}_0e^{-\nu_0(s-s_1)}\|f(s_1)\|_{L^2_{x,v}} ds_1ds.
 \end{align}
Then it follows from \eqref{tdecay1}, \eqref{linearL2} and \eqref{I611} that
\begin{align}\label{I61}
	\dis I_{61}(t,x,v)&\leq C_{k,N,\la}\int^t_0e^{-\nu_0(t-s)}\int^{s-\la}_0e^{-\nu_0(s-s_1)}(1+s_1)^{-\frac{3}{4}}(\|f_0\|_{L^2_{x,v}}+\|f_0\|_{L^2_vL^1_x}) ds_1ds\notag\\
	&\leq C_{k,N,\la}(1+t)^{-\frac{3}{4}}(\|f_0\|_{L^2_{x,v}}+\|f_0\|_{L^2_vL^1_x}).
\end{align}
For Case $4$, it holds by \eqref{HC4} and \eqref{I61} that
\begin{multline}\label{HC40}
	I_6(t,x,v)\leq \frac{C_k}{N}(1+t)^{-\frac{3}{4}}\sup_{0\leq s\leq t}\|(1+s)^\frac{3}{4}w_{k}f(s)\|_{L^\infty_{x,v}}\\
	+C_{k,N,\la}(1+t)^{-\frac{3}{4}}(\|f_0\|_{L^2_{x,v}}+\|f_0\|_{L^2_vL^1_x}).
\end{multline}

Collecting four cases \eqref{HC1}, \eqref{HC2}, \eqref{HC3} and \eqref{HC40}, one has from \eqref{wf1} that
\begin{align*}
	\dis \|w_{k}f(t)\|_{L^\infty_{x,v}}&\leq C_ke^{-\la_0 t}\|w_{k}f_0\|_{L^\infty_{x,v}}\\
	&\quad+ \big(\frac{C_k}{N}+C_k\la\big)(1+t)^{-\frac{3}{4}}\sup_{0\leq s\leq t}\|(1+s)^\frac{3}{4}w_{k}f(s)\|_{L^\infty_{x,v}}\notag\\
	&\quad+C_{k,N,\la}(1+t)^{-\frac{3}{4}}(\|f_0\|_{L^2_{x,v}}+\|f_0\|_{L^2_vL^1_x}).
\end{align*}
Then \eqref{linearLinf} holds by choosing $N$ large and $\la$ small.

The estimate \eqref{linearLinfL2} can be obtained in the similar way. From \eqref{iteration} we have
\begin{align}\label{wf2}
	\|w_{k}(v)f(t,v)\|_{L^2_x}\leq C_ke^{-\la_0 t}\|w_{k}f_0\|_{L^\infty_{v}L^2_x}+\|I_6(t,v)\|_{L^2_x}.
\end{align}
For $\|I_6(t,v)\|_{L^2_x}$, it holds
\begin{align*}
	\|I_6(t,v)\|_{L^2_x}=&\int_0^t e^{-\nu(v)(t-s)}\int_{\R^3}\int_{\R^3}\left|k_{w_k}(v,\eta)k_{w_k}(\eta,\xi)\right|\notag\\
	&\qquad\quad\times\int_0^s e^{-\nu(\eta)(s-s_1)}\big(\int_{\R^3}\left|w_{k}(\xi)f(s_1,y,\xi)\right|^2dy\big)^\frac{1}{2}d\xi ds_1d\eta  ds.
\end{align*}
Notice that similarly to the first three cases in the proof of \eqref{linearLinf}, we can obtain $C_k\sup_{0\leq s\leq t}\|(1+s)^\frac{3}{4}w_{k}f(s)\|_{L^\infty_vL^2_x}$ from the integral. Then the rest part is independent of the space variable so that we can get $\frac{C_k}{N}(1+t)^{-\frac{3}{4}}$ and $C_k \la(1+t)^{-\frac{3}{4}}$ as in \eqref{HC1}, \eqref{HC2} and \eqref{HC3}. From this observation, we consider $\|I_6(t,v)\|_{L^2_x}$ in two cases.

\medskip
\noindent{\it Case 1. } $|v|\geq N$ or $|v|\leq N$, $|\eta|\geq 2N$ or $|\eta|\leq2N$, $|\xi|\geq3N$ or $s-s_1\leq\la$. It holds from similar arguments as in \eqref{HC1}, \eqref{HC2} and \eqref{HC3} that
\begin{align}\label{HC21}
	\dis \|I_6(t,v)\|_{L^2_x}&\leq  \big(\frac{C_k}{N}+C_k\la\big)(1+t)^{-\frac{3}{4}}\sup_{0\leq s\leq t}\|(1+s)^\frac{3}{4}w_{k}f(s)\|_{L^\infty_vL^2_x}.
\end{align}

\medskip
\noindent{\it Case 2. } $|v|\leq N$, $|\eta|\leq2N$, $|\xi|\leq3N$, $s-s_1\geq\la$. We use the same approximation function $k_N$ in \eqref{app} to get
\begin{align}\label{HC22}
	\dis \|I_6(t,v)\|_{L^2_x}&\leq \frac{C_k}{N}(1+t)^{-\frac{3}{4}}\sup_{0\leq s\leq t}\|(1+s)^\frac{3}{4}w_{k}f(s)\|_{L^\infty_vL^2_x}\notag\\
	&\quad+\int^t_0 e^{-\nu_0(t-s)}\iint_{|\eta|\leq2N,|\xi|\leq3N}\left|k_{N}(v,\eta)k_{N}(\eta,\xi)\right| \notag\\
	&\qquad\qquad\times\int^{s-\la}_0e^{-\nu_0(s-s_1)}\big(\int_{\R^3}\left|f(s_1,y,\xi)\right|^2dy\big)^\frac{1}{2}ds_1d\eta d\xi  ds.
\end{align}
Using the property of $k_N$, Cauchy-Schwarz's inequality and \eqref{linearL2}, we have
\begin{align}\label{HC221}
	\dis &\int^t_0 e^{-\nu_0(t-s)}\iint_{|\eta|\leq2N,|\xi|\leq3N}\left|k_{N}(v,\eta)k_{N}(\eta,\xi)\right| \int^{s-\la}_0e^{-\nu_0(s-s_1)}\big(\int_{\R^3}\left|f(s_1,y,\xi)\right|^2dy\big)^\frac{1}{2}ds_1d\eta d\xi  ds\notag\\
	&\leq C_{k,N}\int^t_0 e^{-\nu_0(t-s)}\int^{s-\la}_0e^{-\nu_0(s-s_1)}\int_{|\xi|\leq3N} \big(\int_{\R^3}\left|f(s_1,y,\xi)\right|^2dy\big)^\frac{1}{2}d\xi ds_1 ds\notag\\
	&\leq C_{k,N}\int^t_0 e^{-\nu_0(t-s)}\int^{s-\la}_0e^{-\nu_0(s-s_1)}\big(\int_{|\xi|\leq3N} \int_{\R^3}\left|f(s_1,y,\xi)\right|^2dyd\xi\big)^\frac{1}{2} ds_1 ds\notag\\
	&\leq C_{k,N,\la}(1+t)^{-\frac{3}{4}}(\|f_0\|_{L^2_{x,v}}+\|f_0\|_{L^2_vL^1_x}).
\end{align}
Then for Case $2$, noticing $\|f_0\|_{L^2_{x,v}}\leq C\|w_k f_0\|_{L^\infty_vL^2_x}$ for $k\geq k_0>3$, it holds from \eqref{HC22} and \eqref{HC221} that
\begin{align}\label{HC220}
	\|I_6(t,v)\|_{L^2_x}&\leq  \frac{C_k}{N}(1+t)^{-\frac{3}{4}}\sup_{0\leq s\leq t}\|(1+s)^\frac{3}{4}w_{k}f(s)\|_{L^\infty_vL^2_x}\notag\\
	&\qquad\qquad\qquad\qquad+C_{k,N,\la}(1+t)^{-\frac{3}{4}}(\|w_{k}f_0\|_{L^\infty_vL^2_x}+\|f_0\|_{L^2_vL^1_x}).
\end{align}
We deduce from \eqref{wf2}, \eqref{HC21}, \eqref{HC220} that
\begin{align*}
	\|w_{k}f(t)\|_{L^\infty_vL^2_x}&\leq C_ke^{-\la_0 t}\|w_{k}f_0\|_{L^\infty_{v}L^2_x}\\
	&\quad+ \big(\frac{C_k}{N}+C_k\la\big)(1+t)^{-\frac{3}{4}}\sup_{0\leq s\leq t}\|(1+s)^\frac{3}{4}w_{k}f(s)\|_{L^\infty_vL^2_x}\notag\\
	&\quad+C_{k,N,\la}(1+t)^{-\frac{3}{4}}(\|w_{k}f_0\|_{L^\infty_vL^2_x}+\|f_0\|_{L^2_vL^1_x}).
\end{align*}
Hence, by choosing $N$ large and $\la$ small, one gets \eqref{linearLinfL2} and the proof of Lemma \ref{lem4.2} is then complete.
\end{proof}

\subsection{Estimate on $g_2$}
To take care of the inhomogeneous term in the  equation \eqref{g2}, it is also necessary to estimate the operator $\CK_b$.

\begin{lemma}\label{leKb}
	Let $0\leq \ga\leq 1$ and $k>3+\ga$. Let the operator $\CK_b$ be defined in \eqref{defKb} with the constant $M=M(k)$ chosen in \eqref{chooseM}. Then, for any  $1\leq p\leq \infty$ and $j\geq 0$, it holds that
	\begin{align}
	&\|w_{j} \CK_bf\|_{L^\infty_vL^p_x}\leq C_{j,k}\|w_{k}f\|_{L^\infty_vL^p_x}\label{Kblp}.	
	\end{align}
\end{lemma}
\begin{proof}
	Recalling the definitions of $\CK_b$ in \eqref{defKb} and $\CK$ in \eqref{DefCK}, by the fact that $w_{j}(v)\chi_{\{|v|< M\}}\leq C_{j,k}w_{k}(v)$, one has
\begin{align}\label{Kb2}
	\big\|w_{j}(v)\CK_bf(v)\big\|_{L^p_x}&=\big|\chi_{\{|v|< M\}}\mu^{-1/2}(v)w_{j}(v) \big\|\CK f(v)\big\|_{L^p_x}\big|\notag\\
	 &\leq 	C_{j,k}\|w_{k}f\|_{L^\infty_{v}L^p_x} \chi_{\{|v|< M\}} \notag\\
	 &\quad\times\int_{\R^3}\int_{\S^2}|v-u|^\ga w_k(v)\big( \frac{\mu(u')}{w_{k}(v')}+\frac{\mu(v')}{w_{k}(u')}+\frac{\mu(v)}{w_{k}(u)}\big)d\omega du.
 	\end{align}
By \eqref{controlK} and the fact that $\int_{\R^3}|v-u|^\ga w_k(v)\frac{\mu(v)}{w_{k}(u)}du\leq C_k$, it holds that
\begin{align*}
	\big\|w_{j}(v)\CK_bf(v)\big\|_{L^p_x}&\leq C_{j,k}\|w_{k}f\|_{L^\infty_vL^p_x},
\end{align*}
which yields \eqref{Kblp} and completes the proof of Lemma \ref{leKb}.
\end{proof}

Now, with the properties for the equation \eqref{linearf} and $\CK_b$ obtained above, we deduce the estimates for $g_2$. Denote $U(t)$ to be the solution operator for the problem \eqref{linearf}.
Then, by the Duhamel Principle,
\begin{align}\label{reg2}
g_2(t)=\int_0^tU(t-s)\CK_b g_1(s)\,ds.
\end{align}

\begin{lemma}\label{leg2hard}
	Let $0\leq\gamma\leq 1$, then there is $k_0>0$ large enough such that for any $k\geq k_0$, there is $M>0$ for the decomposition $\CK=\CK_s+\sqrt{\mu}\CK_b$  such that the following estimates hold:
	\begin{align}
		\|w_{k}g_2(t)\|_{L^\infty_{x,v}}&\leq C_k(1+t)^{-\frac{3}{4}}\left\{\|w_kg_{10}\|_{L^\infty_{x,v}}+\|w_{k}g_{10}\|_{L^\infty_vL^2_x}+\|w_{k}g_{10}\|_{L^\infty_vL^1_x}\right\}\notag\\
		&\quad+C_k(1+t)^{-\frac{3}{4}}\left\{\sup_{0\leq s\leq t}\|(1+s)^\frac{3}{4}w_kg(s)\|^2_{L^\infty_{x,v}}\right.\notag\\
		&\qquad\qquad\qquad\qquad\qquad\left.+\sup_{0\leq s\leq t}\|(1+s)^\frac{3}{4}w_{k}g(s)\|^2_{L^\infty_vL^2_x}\right\},\label{g2Linf}
	\end{align}
	and
	\begin{align}
				\|w_{k}g_2(t)\|_{L^\infty_vL^2_x}&\leq C_k(1+t)^{-\frac{3}{4}}\left\{\|w_{k}g_{10}\|_{L^\infty_vL^2_x}+\|w_{k}g_{10}\|_{L^\infty_vL^1_x}\right\}\notag\\
		&\quad+C_k(1+t)^{-\frac{3}{4}}\left\{\sup_{0\leq s\leq t}\|(1+s)^\frac{3}{4}w_kg(s)\|^2_{L^\infty_{x,v}}\right.\notag\\
		&\qquad\qquad\qquad\qquad\qquad\left.+\sup_{0\leq s\leq t}\|(1+s)^\frac{3}{4}w_{k}g(s)\|^2_{L^\infty_vL^2_x}\right\}.\label{g2LinfL2}
	\end{align}
\end{lemma}
\begin{proof}
	We use \eqref{reg2} and \eqref{linearLinf} to get
	\begin{align*}
		|w_{k}g_2(t)|&\leq\int_0^t|w_kU(t-s)(\CK_b g_1)(s)|ds\notag\\
		&\leq C_k\int_0^t(1+t-s)^{-\frac{3}{4}}(\|w_{k}\CK_b g_1(s)\|_{L^\infty_{x,v}}+\|\CK_b g_1(s)\|_{L^2_{x,v}}+\|\CK_b g_1(s)\|_{L^2_vL^1_x})ds.
	\end{align*}
Noticing the fact that $\|f\|_{L^2}\leq C\|w_{k}f\|_{L^\infty}$ for $k\geq k_0>3$, we have
	\begin{align*}
	|w_{k}g_2(t)|\leq C_k\int_0^t(1+t-s)^{-\frac{3}{4}}(\|w_{k}\CK_b g_1(s)\|_{L^\infty_{x,v}}+\|w_{k}\CK_b g_1(s)\|_{L^\infty_vL^2_x}+\|w_{k}\CK_b g_1(s)\|_{L^\infty_vL^1_x})ds.
\end{align*}
 Then it follows from \eqref{Kblp} that
	\begin{align*}
	|w_{k}g_2(t)|&\leq C_k\int_0^t(1+t-s)^{-\frac{3}{4}}(\|w_{k}g_1(s)\|_{L^\infty_{x,v}}+\|w_{k} g_1(s)\|_{L^\infty_vL^2_x}+\|w_{k}g_1(s)\|_{L^\infty_vL^1_x})ds.
\end{align*}
Recalling our estimate for $g_1$ \eqref{g1infty}, \eqref{g1L2} and \eqref{g1L1}, it holds that
	\begin{align}\label{g2infhard}
	|w_{k}g_2(t)|&\leq C_k\int_0^t(1+t-s)^{-\frac{3}{4}}(1+s)^{-\frac{3}{2}}\big(\|w_{k}g_{10}\|_{L^\infty_{x,v}}+\|w_{k} g_{10}\|_{L^\infty_vL^2_x}+\|w_{k}g_{10}\|_{L^\infty_vL^1_x}\big)ds\notag\\
	&\qquad+C_k\int_0^t(1+t-s)^{-\frac{3}{4}}(1+s)^{-\frac{3}{2}}\notag\\
	&\qquad\qquad\qquad\times \left\{\sup_{0\leq s\leq t}\|(1+s)^\frac{3}{4}w_kg(s)\|^2_{L^\infty_{x,v}}+\sup_{0\leq s\leq t}\|(1+s)^\frac{3}{4}w_{k}g(s)\|^2_{L^\infty_vL^2_x}\right\}ds.
\end{align}
Thus, \eqref{g2Linf} follows by \eqref{tdeday} and \eqref{g2infhard}.

In a very similar way, by \eqref{reg2}, \eqref{linearLinfL2}, \eqref{Kblp} \eqref{g1L2} and \eqref{g1L1}, we have
	\begin{align*}
	\|w_{k}g_2(t)\|_{L^\infty_vL^2_x}&\leq\int_0^t\|w_kU(t-s)(\CK_b g_1)(s)\|_{L^\infty_vL^2_x}ds\notag\\
	&\leq C_k\int_0^t(1+t-s)^{-\frac{3}{4}}(\|w_{k}\CK_b g_1(s)\|_{L^\infty_vL^2_x}+\|w_{k}\CK_b g_1(s)\|_{L^\infty_vL^1_x})ds\notag\\
	&\leq C_k\int_0^t(1+t-s)^{-\frac{3}{4}}(1+s)^{-\frac{3}{2}}\big(\|w_{k} g_{10}\|_{L^\infty_vL^2_x}+\|w_{k}g_{10}\|_{L^\infty_vL^1_x}\big)ds\notag\\
	&\qquad+C_k\int_0^t(1+t-s)^{-\frac{3}{4}}(1+s)^{-\frac{3}{2}}\notag\\
	&\qquad\qquad\times \left\{\sup_{0\leq s\leq t}\|(1+s)^\frac{3}{4}w_kg(s)\|^2_{L^\infty_{x,v}}+\sup_{0\leq s\leq t}\|(1+s)^\frac{3}{4}w_{k}g(s)\|^2_{L^\infty_vL^2_x}\right\}ds.
\end{align*}
Then one obtains \eqref{g2LinfL2}. The proof of Lemma \ref{leg2hard} is complete.
\end{proof}

\subsection{Proof of Theorem \ref{hard}}
With all the preparations, we are able to prove Theorem \ref{hard}.
\begin{proof}[Proof of Theorem \ref{hard}]
	Recall our definition for $g$ that $g=g_1+\sqrt{\mu}g_2$. Then it is straightforward to see that for $k\geq j\geq k_0$,
	\begin{align*}
		\|w_kg(t)\|_{L^\infty_{x,v}}&\leq \|w_kg_1(t)\|_{L^\infty_{x,v}}+\|w_k\sqrt{\mu}g_2(t)\|_{L^\infty_{x,v}}\notag\\
		&\leq \|w_kg_1(t)\|_{L^\infty_{x,v}}+C_{j,k}\|w_{j}g_2(t)\|_{L^\infty_{x,v}}.
	\end{align*}
We have from \eqref{g1infty}, \eqref{g2Linf} and the condition $j\leq k$ that
\begin{align}\label{gLinf}
	\|w_kg(t)\|_{L^\infty_{x,v}}&\leq C_{j,k}(1+t)^{-\frac{3}{4}}\left\{\|w_kg_{10}\|_{L^\infty_{x,v}}+\|w_{j}g_{10}\|_{L^\infty_vL^2_x}+\|w_{j}g_{10}\|_{L^\infty_vL^1_x}\right\}\notag\\
	&\quad+C_{j,k}(1+t)^{-\frac{3}{4}}\left\{\sup_{0\leq s\leq t}\|(1+s)^\frac{3}{4}w_kg(s)\|^2_{L^\infty_{x,v}}+\sup_{0\leq s\leq t}\|(1+s)^\frac{3}{4}w_{j}g(s)\|^2_{L^\infty_vL^2_x}\right\}.
\end{align}
Due to the $L^\infty_vL^2_x$ norm in \eqref{gLinf}, we need to estimate $\sup_{0\leq s\leq t}\|(1+s)^\frac{3}{4}w_{j}g(s)\|^2_{L^\infty_vL^2_x}$. By \eqref{g1L2}, \eqref{g2LinfL2} and the condition $j\leq k$, one has
\begin{align}\label{gLinfL2}
	\|w_{j}g(t)\|_{L^\infty_vL^2_x}&\leq C_j(1+t)^{-\frac{3}{4}}\left\{\|w_{j}g_{10}\|_{L^\infty_vL^2_x}+\|w_{j}g_{10}\|_{L^\infty_vL^1_x}\right\}\notag\\
	&\quad+C_j(1+t)^{-\frac{3}{4}}\left\{\sup_{0\leq s\leq t}\|(1+s)^\frac{3}{4}w_kg(s)\|^2_{L^\infty_{x,v}}+\sup_{0\leq s\leq t}\|(1+s)^\frac{3}{4}w_{j}g(s)\|^2_{L^\infty_vL^2_x}\right\}.
\end{align}
Combining \eqref{gLinf} and \eqref{gLinfL2}, we obtain
\begin{align}\label{ghard}
	&\sup_{0\leq s\leq t}\|(1+s)^\frac{3}{4}w_kg(s)\|_{L^\infty_{x,v}}+\sup_{0\leq s\leq t}\|(1+s)^\frac{3}{4}w_{j}g(s)\|_{L^\infty_vL^2_x}\notag\\
	&\leq C_{j,k}\left\{\|w_kg_{10}\|_{L^\infty_{x,v}}+\|w_{j}g_{10}\|_{L^\infty_vL^2_x}+\|w_{j}g_{10}\|_{L^\infty_vL^1_x}\right\}\notag\\
	&\quad+C_{j,k}\left\{\sup_{0\leq s\leq t}\|(1+s)^\frac{3}{4}w_kg(s)\|^2_{L^\infty_{x,v}}+\sup_{0\leq s\leq t}\|(1+s)^\frac{3}{4}w_{j}g(s)\|^2_{L^\infty_vL^2_x}\right\}.
\end{align}
Recall our definitions of $\|\cdot\|_{X_{j,k}}$ in \eqref{defX} and $\|\cdot\|_{Y_{j,k}}$ in \eqref{defY}. By the local-in-time existence together with the continuity argument, from \eqref{ghard}, \eqref{HE} follows by \eqref{smallnesshard} for a small constant $\ep_0$ which depends on $j$ and $k$. Hence,  the global existence is established and the proof of Theorem \ref{hard} is complete.
\end{proof}

\section{Soft potential case}\label{sec5}
Compared to hard potential case, it is more complicated to prove Theorem \ref{soft} in soft potentials. One of the main difficulties in this case is that \eqref{decay} holds only for $0\leq\ga\leq1$ and we have to use \eqref{decays} instead, which requires us to choose the index $r$ in \eqref{decays} more carefully. Furthermore, we still need to analyze the equation \eqref{linearf}. In soft potential case, the large velocity decay that the operator $K$ provides is not strong enough, so we have to split it into $K^m$ and $K^c$ as mentioned before. Also, the collision frequency $\nu(v)\sim(1+|v|)^\ga$ has no strictly positive lower bound and the spectral gap of the linearized operator vanishes when $-3<\ga<0$.

To carry out the proof in soft potentials, we still start from the decomposed sytem \eqref{g1} and \eqref{g2} with the initial data \eqref{initial}.

\subsection{Estimates on $g_1$}
We study $g_1$ in the similar way as in Lemma \ref{leg1hard}, but for $-3<\ga<0$, there are more restrictions than the previous lemma.

\begin{lemma}\label{leg1soft}
	Let $-3<\ga<0$ and $0<\ep\leq\frac{1}{2}$. There is $k_0>0$ large enough such that for any $k\geq k_0$, there is a constant $M>0$ for the decomposition $\CK=\CK_s+\sqrt{\mu}\CK_b$  such that the following estimates hold:
	\begin{align}
		\|w_kg_1(t)\|_{L^\infty_{x,v}}\leq &C_\ep(1+t)^{-1+\ep}\|w_{k+|\ga|}g_{10}\|_{L^\infty_{x,v}}\notag\\
		&\qquad\qquad\quad\qquad+C_{\ep,k}(1+t)^{-1+\ep}\sup_{0\leq s\leq t}\|(1+s)^{\frac{3}{4}-\ep}w_kg(s)\|^2_{L^\infty_{x,v}},\label{sg1infty}\\
		\|w_{k}g_1(t)\|_{L^\infty_vL^2_x}&\leq C_\ep(1+t)^{-1+\ep}\|w_{k+|\ga|}g_{10}\|_{L^\infty_vL^2_x}\notag\\
		+C_{\ep,k}&(1+t)^{-1+\ep}\left\{\sup_{0\leq s\leq t}\|(1+s)^{\frac{3}{4}-\ep}w_kg(s)\|^2_{L^\infty_{x,v}}+\sup_{0\leq s\leq t}\|(1+s)^{\frac{3}{4}-\ep}w_{k}g(s)\|^2_{L^\infty_vL^2_x}\right\},\label{sg1L2}\\
		\|w_{k}g_1(t)\|_{L^\infty_vL^1_x}&\leq C_\ep(1+t)^{-1+\ep}\|w_{k+|\ga|}g_{10}\|_{L^\infty_vL^1_x}\notag\\
		&\qquad\qquad\quad\qquad+C_{\ep,k}(1+t)^{-1+\ep}\sup_{0\leq s\leq t}\|(1+s)^{\frac{3}{4}-\ep}w_{k}g(s)\|^2_{L^\infty_vL^2_x}\label{sg1L1},
	\end{align}
where $C_\ep$ depends only on $\ep$ and $C_{\ep,k}$ depends only on $\ep$ and $k$.
\end{lemma}
\begin{proof}
	We first choose $k_0>3$ such that Lemma \ref{leKs} and Lemma \ref{leQ} hold for $k\geq k_0$. The mild form of $g_1$ \eqref{mildg1} is still valid. Using the fact that
	\begin{align}\label{estesoft}
	e^{-\nu(v)t}\leq \frac{C_\ep}{|\nu(v)(1+t)|^{1-\ep}},
	\end{align}
	one gets that
	\begin{align}\label{sg11}
		\dis \left|w_k(v)g_1(t,x,v)\right|\leq&C_\ep(1+t)^{-1+\ep}\|w_{k+|\ga|}g_{10}\|_{L^\infty_{x,v}}\notag\\
		&+\int_0^t e^{-\nu(v)(t-s)}\left|w_k(v)\CK_s g_1(s,x-v(t-s),v)\right|ds \notag\\
		&+\int_0^t e^{-\nu(v)(t-s)}\left|w_k(v)Q(g,g)(s,x-v(t-s),v)\right|ds\notag\\
		=&C_\ep(1+t)^{-1+\ep}\|w_{k+|\ga|}g_{10}\|_{L^\infty_{x,v}}+J_1(t,x,v)+J_2(t,x,v).
	\end{align}
By \eqref{smallCK} and \eqref{decays}, it holds that
\begin{align}\label{C2ep}
	\|J_1(t,v)\|_{L^\infty_x}\leq &C\sup_{0\leq s\leq t}\|(1+s)^{1-\ep}w_k g_1(s)\|_{L^\infty_{x,v}}\int_0^t e^{-\nu(v)(t-s)}\nu(v)(1+s)^{-1+\ep}\left(\frac{C}{k^{\frac{\ga+3}{2}}}+\frac{C_k}{M^2}\right)ds\notag\\
	\leq &C_{2\ep}\left(\frac{1}{k^{\frac{\ga+3}{2}}}+\frac{C_k}{M^2}\right)	(1+t)^{-1+\ep}\sup_{0\leq s\leq t}\|(1+s)^{1-\ep}w_k g_1(s)\|_{L^\infty_{x,v}},
\end{align}
where $C_{2\ep}$ depends only on $\ep$.
Let $k_0>(4C_{2\ep})^\frac{2}{\ga+3}$, it holds for $k\geq k_0$ that
$$
\frac{C_{2\ep}}{k^{\frac{\ga+3}{2}}}\leq \frac{1}{4}.
$$
Then we define
\begin{align}\label{Mep}
	M=2\sqrt{C_{2\ep}C_k}
\end{align}
to obtain
\begin{align}\label{J1}
	\|J_1(t,v)\|_{L^\infty_x}\leq\frac{1}{2}(1+t)^{-1+\ep}\sup_{0\leq s\leq t}\|(1+s)^{1-\ep}w_k g_1(s)\|_{L^\infty_{x,v}}.
\end{align}
For $J_2(t,x,v)$, using \eqref{controlQ} and \eqref{decays}, similar arguments as in \eqref{I51} and \eqref{I5} show that
\begin{align}\label{J2}
	|J_2(t,x,v)|&\leq \int_0^t e^{-\nu(v)(t-s)}\left|w_k(v)Q(g,g)(s,x-v(t-s),v)\right|ds\notag\\
	&\leq C\sup_{0\leq s\leq t}\|(1+s)^{\frac{3}{4}-\ep}w_kg(s)\|^2_{L^\infty_{x,v}}\int_0^t e^{-\nu(v)(t-s)}(1+s)^{-\frac{3}{2}+2\ep}\notag\\
	&\qquad\qquad\qquad\qquad\qquad\times\int_{\R^3}\int_{\S^2}|v-u|^\ga\left( \frac{w_k(v)}{w_k(u')w_k(v')}+\frac{1}{w_k(u)}\right) d\omega duds\notag\\
	&\leq C_{\ep,k}(1+t)^{-1+\ep}\sup_{0\leq s\leq t}\|(1+s)^{\frac{3}{4}-\ep}w_kg(s)\|^2_{L^\infty_{x,v}}.
\end{align}
The last inequality above holds by the fact that $(1+s)^{-\frac{3}{2}+2\ep}\leq C(1+s)^{-1+\ep}$ under the condition $0<\ep\leq \frac{1}{2}$. Collecting \eqref{sg11}, \eqref{J1} and \eqref{J2}, we have
\begin{align*}%\label{g1infty2}
	\dis \|w_kg_1(t)\|_{L^\infty_{x,v}}\leq&C_\ep(1+t)^{-1+\ep}\|w_{k+|\ga|}g_{10}\|_{L^\infty_{x,v}}+\frac{1}{2}(1+t)^{-1+\ep}\sup_{0\leq s\leq t}\|(1+s)^{1-\ep}w_k g_1(s)\|_{L^\infty_{x,v}} \notag\\
	&+C_{\ep,k}(1+t)^{-1+\ep}\sup_{0\leq s\leq t}\|(1+s)^{\frac{3}{4}-\ep}w_kg(s)\|^2_{L^\infty_{x,v}},
\end{align*}
which yields \eqref{sg1infty}.

Similarly as in \eqref{wk0g1L2}, \eqref{I5L2}, \eqref{C2ep} and \eqref{J1}, one gets
 \begin{align*}
	\dis \|w_{k}g_1(t)\|_{L^\infty_vL^2_x}\leq&C_\ep(1+t)^{-1+\ep}\|w_{k+|\ga|}g_{10}\|_{L^\infty_vL^2_x}\notag\\
	&+\frac{1}{2}(1+t)^{-1+\ep}\sup_{0\leq s\leq t}\|(1+s)^{1-\ep}w_{k} g_1(s)\|_{L^\infty_vL^2_x}\notag\\
	&+C\int_0^t e^{-\nu(v)(t-s)}(1+s)^{-1+\ep}\int_{\R^3}\int_{\S^2}|v-u|^\ga\left(\frac{w_{k}(v)}{w_{k}(v')w_{k}(u')}+\frac{1}{w_{k}(u)}\right)\notag\\
	&\qquad\quad\times \|(1+s)^{\frac{3}{4}-\ep}w_{k}g(s)\|_{L^\infty_vL^2_x}\|(1+s)^{\frac{3}{4}-\ep}w_{k}g(s)\|_{L^\infty_{x,v}}d\omega duds.
\end{align*}	
Then using Cauchy-Schwarz's inequality, \eqref{controlQ} and \eqref{decays} as in \eqref{I5L2}, we can obtain \eqref{sg1L2} in the following way:	
 \begin{align*}
\dis \|w_{k}g_1(t)\|_{L^\infty_vL^2_x}	\leq&C_\ep(1+t)^{-1+\ep}\|w_{k+(1-\ep)|\ga|}g_{10}\|_{L^\infty_vL^2_x}\notag\\
	&+\frac{1}{2}(1+t)^{-1+\ep}\sup_{0\leq s\leq t}\|(1+s)^{1-\ep}w_{k} g_1(s)\|_{L^\infty_vL^2_x}\notag\\
	&+C_{\ep,k}(1+t)^{-1+\ep}\notag\\
	&\ \times\left\{\sup_{0\leq s\leq t}\|(1+s)^{\frac{3}{4}-\ep}w_kg(s)\|^2_{L^\infty_{x,v}}+\sup_{0\leq s\leq t}\|(1+s)^{\frac{3}{4}-\ep}w_{k}g(s)\|^2_{L^\infty_vL^2_x}\right\}.
\end{align*}

At last, from similar arguments for deriving \eqref{I5L1}, it follows that
 \begin{align*}
	\dis \|w_{k}g_1(t)\|_{L^\infty_vL^1_x}\leq&C_\ep(1+t)^{-1+\ep}\|w_{k+(1-\ep)|\ga|}g_{10}\|_{L^\infty_vL^1_x}\notag\\
	&+\frac{1}{2}(1+t)^{-1+\ep}\sup_{0\leq s\leq t}\|(1+s)^{1-\ep}w_{k} g_1(s)\|_{L^\infty_vL^1_x}\notag\\
	&+C_{\ep,k}(1+t)^{-1+\ep}\sup_{0\leq s\leq t}\|(1+s)^{\frac{3}{4}-\ep}w_{k}g(s)\|^2_{L^\infty_vL^2_x},
\end{align*}
which indicates \eqref{sg1L1}. The proof of Lemma \ref{leg1soft} is complete.
\end{proof}

\subsection{Time decay in the symmetric case}
To estimate $g_2$, we still start with the form \eqref{reg2}, which requires us to study $U(t)$ and $\CK_b$ for $-3<\ga<0$. In this case, we first give the $L^2$ decay property of the equation \eqref{linearf} in the following proposition. Notice we need additional velocity weight on the initial data, cf. \cite[Theorem 4.1, pp.23]{DL-cmp}. Again, as for Proposition \ref{propL2}, we omit the proof for brevity; see \cite{DSa} and references therein, for instance.

\begin{proposition}\label{propL2s}
	Let $-3<\ga<0$. Let $f$ be the solution to the problem \eqref{linearf}, then it holds
	\begin{align}\label{linearL2s}
	\|f(t)\|_{L^2_{x,v}}\leq C(1+t)^{-\frac{3}{4}}(\|\nu^{-1} f_0\|_{L^2_{x,v}}+\|\nu^{-1}f_0\|_{L^2_vL^1_x}),
	\end{align}
	for any $t\geq0$.
\end{proposition}
With this proposition, we are able to obtain the $L^\infty_v$ decay properties of the solution to \eqref{linearf}.

\begin{lemma}\label{lem5.2}
	Let $-3<\ga<0$ and $k> 3$. Let $f$ be the solution to the problem \eqref{linearf}, then it holds
	\begin{align}
		&\|w_{k}f(t)\|_{L^\infty_{x,v}}\leq C_k(1+t)^{-\frac{3}{4}}(\|w_{k+|\ga|}f_0\|_{L^\infty_{x,v}}+\|\nu^{-1}f_0\|_{L^2_{x,v}}+\|\nu^{-1}f_0\|_{L^2_vL^1_x}), \label{slinearLinf}\\
		&\|w_{k}f(t)\|_{L^\infty_vL^2_x}\leq C_k(1+t)^{-\frac{3}{4}}(\|w_{k+|\ga|}f_0\|_{L^\infty_vL^2_x}+\|\nu^{-1}f_0\|_{L^2_vL^1_x}), \label{slinearLinfL2}
	\end{align}
	for any $t\geq0$.
\end{lemma}
\begin{proof}
	Since for soft potentials we need more decay on $v$ from the operator $K$, we split $K$ into $K=K^m+K^c$ where $K^m$ is defined in \eqref{defKm}. Rewrite
	\begin{align}\label{milds}
		\dis (w_k f)(t,x,v)&=e^{-\nu(v)t} (w_k f_0)(x-vt,v)+\int_0^t e^{-\nu(v)(t-s)}(w_k K^mf)(s,x-v(t-s),v)ds \notag\\
		&\quad+\int_0^t e^{-\nu(v)(t-s)}(w_k K^cf)(s,x-v(t-s),v)ds.
	\end{align}
By \eqref{estesoft}, a direct calculation shows that
\begin{align}\label{f0s}
	|e^{-\nu(v)t} (w_k f_0)(x-vt,v)|\leq C(1+t)^{-\frac{3}{4}}\|w_{k+|\ga|}f_0\|_{L^\infty_{x,v}}.
\end{align}
Then by \eqref{smallKm} and \eqref{decays}, we have
\begin{align}\label{Kms}
	&\int_0^t e^{-\nu(v)(t-s)}\left|(w_k K^mf)(s,x-v(t-s),v)\right|ds\notag\\
	&\leq Cm^{\ga+3}\sup_{0\leq s\leq t}\|(1+s)^\frac{3}{4}w_{k}f(s)\|_{L^\infty_{x,v}}w_{k+|\ga|}(v) e^{-\frac{|v|^2}{10}} \int^t_0e^{-\nu(v)(t-s)}\nu(v)(1+s)^{-\frac{3}{4}}ds  \notag\\
	&\leq C_km^{\ga+3}(1+t)^{-\frac{3}{4}}\sup_{0\leq s\leq t}\|(1+s)^\frac{3}{4}w_{k}f(s)\|_{L^\infty_{x,v}}.
\end{align}
As in \eqref{Defkw}, we define
	\begin{align*}
	l_{w_k}(v,\eta)=l(v,\eta)\frac{w_k(v)}{w_k(\eta)}.
\end{align*}
It follows from \eqref{RepKc}, \eqref{milds}, \eqref{f0s} and \eqref{Kms} that
\begin{align}\label{milds1}
	\dis |(w_k f)(t,x,v)|&\leq C(1+t)^{-\frac{3}{4}}\|w_{k+|\ga|}f_0\|_{L^\infty_{x,v}}\notag\\
	&\qquad\quad+C_km^{\ga+3}(1+t)^{-\frac{3}{4}}\sup_{0\leq s\leq t}\|(1+s)^\frac{3}{4}w_{k}f(s)\|_{L^\infty_{x,v}} +J_3(t,x,v),
\end{align}
where
\begin{align*}
	J_3(t,x,v)&=\int_0^t e^{-\nu(v)(t-s)}\int_{\R^3}w_k(v)\big|l(v,\eta)f(s,x-v(t-s),\eta)\big|d\eta ds\notag\\
	&=\int_0^t e^{-\nu(v)(t-s)}\int_{\R^3}\big|l_{w_k}(v,\eta)w_k(\eta) f(s,x_1,\eta)\big|d\eta ds,
\end{align*}
with $x_1=x-v(t-s)$.
Again using \eqref{milds}, one gets
\begin{align}\label{J3}
	\dis J_3(t,x,v)\leq& \int^t_0 e^{-\nu(v)(t-s)}\int_{\R^3}|l_{w_k}(v,\eta)|e^{-\nu(\eta)s}|(w_k f_0)(x_1-\eta s,\eta)|d\eta ds\notag\\
	&+\int^t_0 e^{-\nu(v)(t-s)}\int_{\R^3}|l_{w_k}(v,\eta)|\int^s_0 e^{-\nu(\eta)(s-s_1)}\big|\left(w_k K^mf\right)(s_1,x_1-\eta(s-s_1),\eta)\big|ds_1d\eta ds\notag\\
	&+\int^t_0 e^{-\nu(v)(t-s)}\int_{\R^3}\int_{\R^3}|l_{w_k}(v,\eta)l_{w_k}(\eta,\xi)|\notag\\
	&\quad\times \int^s_0e^{-\nu(\eta)(s-s_1)}|(w_k f)(s_1,x_1-\eta(s-s_1),\xi)|ds_1d\eta d\xi ds\notag\\
	=&J_{31}(t,x,v)+J_{32}(t,x,v)+J_{33}(t,x,v).
\end{align}
By \eqref{estesoft}, it is straightforward to see that
\begin{align}\label{J31}
	\dis J_{31}(t,x,v)\leq& \int^t_0 e^{-\nu(v)(t-s)}\int_{\R^3}\big|l_{w_k}(v,\eta)e^{-\nu(\eta)s}(w_k f_0)(x_1-\eta s,\eta)\big|d\eta ds\notag\\
	\leq &C\int^t_0 e^{-\nu(v)(t-s)}\int_{\R^3}|l_{w_k}(v,\eta)|\frac{1}{|\nu(\eta)(1+s)|^\frac{3}{4}}|(w_k f_0)(x_1-\eta s,\eta)|d\eta ds\notag\\
	\leq &C\|w_{k+|\ga|}f_0\|_{L^\infty_{x,v}}\int^t_0 e^{-\nu(v)(t-s)}\nu(v)(1+s)^{-\frac{3}{4}}\int_{\R^3}|l_{w_k}(v,\eta)|\frac{1}{\nu(v)}d\eta ds.
\end{align}
Notice that from \eqref{decayl} we have
$$
\int_{\R^3}|l_{w_k}(v,\eta)|\frac{1}{\nu(v)}d\eta\leq C_{m,k}\frac{\nu(v)}{(1+|v|)^2}\frac{1}{\nu(v)}\leq C_{m,k},
$$
which together with \eqref{decays} and \eqref{J31}, yields
\begin{align}\label{J310}
	\dis J_{31}(t,x,v)\leq &C_{m,k}(1+t)^{-\frac{3}{4}}\|w_{k+|\ga|}f_0\|_{L^\infty_{x,v}}.
\end{align}
Using \eqref{smallKm} and similar arguments as in \eqref{J31}, we have
\begin{align}\label{J321}
	\dis J_{32}(t,x,v)\leq& \int^t_0 e^{-\nu(v)(t-s)}\int_{\R^3}|l_{w_k}(v,\eta)|\int^s_0 e^{-\nu(\eta)(s-s_1)}\big|\left(w_k K^mf\right)(s_1,x_1-\eta(s-s_1),\eta)\big|ds_1d\eta ds\notag\\
	\leq& Cm^{\ga+3}\sup_{0\leq s\leq t}\|(1+s)^\frac{3}{4}w_{k}f(s)\|_{L^\infty_{x,v}}\notag\\
	&\qquad\times\int^t_0 e^{-\nu(v)(t-s)}\int_{\R^3}|l_{w_k}(v,\eta)|\int^s_0 e^{-\nu(\eta)(s-s_1)}(1+s_1)^{-\frac{3}{4}}w_k(v)e^{-\frac{|v|^2}{10}}ds_1d\eta ds.
\end{align}
It is noted that
\begin{align}\label{J322}
&\int^t_0 e^{-\nu(v)(t-s)}\int_{\R^3}|l_{w_k}(v,\eta)|\int^s_0 e^{-\nu(\eta)(s-s_1)}(1+s_1)^{-\frac{3}{4}}w_k(v)e^{-\frac{|v|^2}{10}}ds_1d\eta ds\notag\\
&\leq C\int^t_0 e^{-\nu(v)(t-s)}\nu(v)(1+s)^{-\frac{3}{4}}\frac{w_k(v)e^{-\frac{|v|^2}{10}}}{\nu(v)^2}\int_{\R^3}|l_{w_k}(v,\eta)|\frac{\nu(v)}{\nu(\eta)} d\eta ds\notag\\
&\leq C_k(1+t)^{-\frac{3}{4}},
\end{align}
where we have used \eqref{slowl} to control $\int_{\R^3}|l_{w_k}(v,\eta)|\frac{\nu(v)}{\nu(\eta)} d\eta$ so the constant $C_k$ is independent of $m$.
Hence, we have from \eqref{J321} and \eqref{J322} that
\begin{align}\label{J32}
	\dis J_{32}(t,x,v)\leq C_km^{\ga+3}(1+t)^{-\frac{3}{4}}\sup_{0\leq s\leq t}\|(1+s)^\frac{3}{4}w_{k}f(s)\|_{L^\infty_{x,v}}.
\end{align}

For $J_{33}(t,x,v)$, we still consider it in four cases.

\noindent{\it Case 1. } $|v|\geq N$. A direct calculation shows that
\begin{align*}
	J_{33}(t,x,v)&\leq C\sup_{0\leq s\leq t}\|(1+s)^\frac{3}{4}w_{k}f(s)\|_{L^\infty_{x,v}}\notag\\
	\times&\int_0^t e^{-\nu(v)(t-s)}\int_{\R^3}\int_{\R^3}\left|l_{w_k}(v,\eta)l_{w_k}(\eta,\xi)\right|\frac{1}{\nu(\eta)}\int_0^s e^{-\nu(\eta)(s-s_1)}\nu(\eta)(1+s_1)^{-\frac{3}{4}}ds_1d\xi d\eta  ds,
\end{align*}
which, together with \eqref{decays} and \eqref{decayl}, yields
 \begin{align}\label{1J33}
 	J_{33}(t,x,v)&\leq C_{m,k}\sup_{0\leq s\leq t}\|(1+s)^\frac{3}{4}w_{k}f(s)\|_{L^\infty_{x,v}}\int_0^t e^{-\nu(v)(t-s)}\nu(v)(1+s)^{-\frac{3}{4}}\int_{\R^3}\left|l_{w_k}(v,\eta)\right|\frac{1}{\nu(v)}d\eta  ds\notag\\
 	&\leq \frac{C_{m,k}}{(1+|v|)^2}(1+t)^{-\frac{3}{4}}\sup_{0\leq s\leq t}\|(1+s)^\frac{3}{4}w_{k}f(s)\|_{L^\infty_{x,v}}\notag\\
 	&\leq \frac{C_{m,k}}{N}(1+t)^{-\frac{3}{4}}\sup_{0\leq s\leq t}\|(1+s)^\frac{3}{4}w_{k}f(s)\|_{L^\infty_{x,v}}.
 \end{align}

 \medskip
 \noindent{\it Case 2. } $|v|\leq N$, $|\eta|\geq 2N$ or $|\eta|\leq2N$, $|\xi|\geq3N$. Similar arguments as in \eqref{HC2} and \eqref{1J33} show that
  \begin{align}\label{2J33}
	J_{33}(t,x,v)&\leq \frac{C}{N}\sup_{0\leq s\leq t}\|(1+s)^\frac{3}{4}w_{k}f(s)\|_{L^\infty_{x,v}}\int_0^t e^{-\nu(v)(t-s)}\notag\\
	\times&\int_{\R^3}\int_{\R^3}\left|l_{w_k}(v,\eta)e^{\frac{|v-\eta|^2}{20}}l_{w_k}(\eta,\xi)e^{\frac{|\eta-\xi|^2}{20}}\frac{1}{\nu(\eta)}\right|\int_0^s e^{-\nu(\eta)(s-s_1)}\nu(\eta)(1+s_1)^{-\frac{3}{4}}ds_1d\xi d\eta  ds\notag\\
	&\leq \frac{C_{m,k}}{N}(1+t)^{-\frac{3}{4}}\sup_{0\leq s\leq t}\|(1+s)^\frac{3}{4}w_{k}f(s)\|_{L^\infty_{x,v}}.
 \end{align}

 \medskip
 \noindent{\it Case 3. } $s-s_1\leq\la$. Similarly, one has
 \begin{align}\label{3J33}
 J_{33}(t,x,v)&\leq C\sup_{0\leq s\leq t}\|(1+s)^\frac{3}{4}w_{k}f(s)\|_{L^\infty_{x,v}}\int_0^t e^{-\nu(v)(t-s)}\notag\\
 	&\qquad\qquad\times\int_{\R^3}\int_{\R^3}\left|l_{w_k}(v,\eta)l_{w_k}(\eta,\xi)\frac{1}{\nu(\eta)}\right|\int_0^s e^{-\nu(\eta)(s-s_1)}\nu(\eta)(1+s_1)^{-\frac{3}{4}}ds_1d\xi d\eta  ds\notag\\
 	&\leq C_{m,k}\la(1+t)^{-\frac{3}{4}}\sup_{0\leq s\leq t}\|(1+s)^\frac{3}{4}w_{k}f(s)\|_{L^\infty_{x,v}}.
 \end{align}

\medskip
\noindent{\it Case 4. } $|v|\leq N$, $|\eta|\leq2N$, $|\xi|\leq3N$, $s-s_1\geq\la$.
We choose a smooth function $l_N$ with compact support such that
\begin{align}\label{deflN}
	\sup_{|v|\leq3N}\int_{|\eta|\leq3N}\left| l_{w_k}(v,\eta)-l_N(v,\eta) \right|d\eta \leq \frac{C_{m,k}}{N^7}.
\end{align}
Then we have
\begin{equation}
\label{4J33}
\dis J_{33}(t,x,v)\leq  J_{331}+J_{332}+J_{333},
\end{equation}
with
\begin{align*}
	J_{331}&:= C\sup_{0\leq s\leq t}\|(1+s)^\frac{3}{4}w_{k}f(s)\|_{L^\infty_{x,v}}\int^t_0 e^{-\nu(v)(t-s)}\notag\\
	&\qquad\times \int_{\R^3}\int_{\R^3}\left|l_{w_k}(v,\eta)-l_{N}(v,\eta)\right|\left|l_{w_k}(\eta,\xi)\right|\int^{s-\la}_0e^{-\nu(\eta)(s-s_1)}(1+s_1)^{-\frac{3}{4}}ds_1d\eta d\xi ds,
\end{align*}
\begin{align*}
	J_{332}&:=C\sup_{0\leq s\leq t}\|(1+s)^\frac{3}{4}w_{k}f(s)\|_{L^\infty_{x,v}}\int^t_0 e^{-\nu(v)(t-s)}\notag\\
	&\qquad\times \int_{\R^3}\int_{\R^3}\left|l_{N}(v,\eta)\right|\left|l_{w_k}(\eta,\xi)-l_{N}(\eta,\xi)\right|\int^{s-\la}_0e^{-\nu(\eta)(s-s_1)}(1+s_1)^{-\frac{3}{4}}ds_1d\eta d\xi ds,
\end{align*}
and
\begin{align*}
	J_{333}:=&\int^t_0 e^{-\nu(v)(t-s)}\iint_{|\eta|\leq2N,|\xi|\leq3N}\left|l_{N}(v,\eta)l_{N}(\eta,\xi)\right|\notag\\
	&\quad\times \int^{s-\la}_0e^{-\nu(\eta)(s-s_1)}\left|(w_k f)(s_1,x_1-\eta(s-s_1),\xi)\right|ds_1d\eta d\xi  ds.
\end{align*}
For $J_{331}$, noticing that $\frac{1}{\nu(v)}\leq C(1+|v|)^{-\ga}\leq CN^3$ and $\frac{1}{\nu(\eta)}\leq CN^3$, by \eqref{decays} and \eqref{deflN}, we have
\begin{align}\label{J331}
	J_{331}&\leq CN^6\sup_{0\leq s\leq t}\|(1+s)^\frac{3}{4}w_{k}f(s)\|_{L^\infty_{x,v}}\int^t_0 e^{-\nu(v)(t-s)}\nu(v)\notag\\
	&\quad\times \int_{\R^3}\int_{\R^3}\left|l_{w_k}(v,\eta)-l_{N}(v,\eta)\right|\left|l_{w_k}(\eta,\xi)\right|\int^{s-\la}_0e^{-\nu(\eta)(s-s_1)}\nu(\eta)(1+s_1)^{-\frac{3}{4}}ds_1d\eta d\xi ds\notag\\
	&\leq \frac{C_{m,k}}{N}(1+t)^{-\frac{3}{4}}\sup_{0\leq s\leq t}\|(1+s)^\frac{3}{4}w_{k}f(s)\|_{L^\infty_{x,v}}.
\end{align}
A similar calculation yields
 \begin{align}\label{J332}
 	J_{332}&\leq \frac{C_{m,k}}{N}(1+t)^{-\frac{3}{4}}\sup_{0\leq s\leq t}\|(1+s)^\frac{3}{4}w_{k}f(s)\|_{L^\infty_{x,v}}.
 \end{align}
Using the arguments in \eqref{I611} and \eqref{I61} and denoting
\begin{align}\label{defnuN}
\nu_N=\inf_{|v|\leq 3N}|\nu(v)|>0,
\end{align}
 it follows from \eqref{defnuN}, \eqref{tdecay1} and \eqref{linearL2s} that
 \begin{align}\label{J333}
	\dis J_{333}(t,x,v)& \leq C_{m,k,N}\int^t_0e^{-\nu_N(t-s)}\int^{s-\la}_0e^{-\nu_N(s-s_1)}\notag\\
	&\qquad\quad\times\left(\iint_{|\eta|\leq2N,|\xi|\leq3N}\left|f(s_1,x_1-\eta(s-s_1),\xi)\right|^2d\eta d\xi\right)^\frac{1}{2}  ds_1ds\notag\\
	&\leq C_{m,k,N,\la}\int^t_0e^{-\nu_N(t-s)}\int^{s-\la}_0e^{-\nu_N(s-s_1)}\left(\int_{\R^3}\int_{\R^3}\left|f(s_1,y,\xi)\right|^2dy d\xi\right)^\frac{1}{2}  ds_1ds\notag\\
	&\leq C_{m,k,N,\la}(1+t)^{-\frac{3}{4}}(\|\nu^{-1}f_0\|_{L^2_{x,v}}+\|\nu^{-1}f_0\|_{L^2_vL^1_x}).
\end{align}
For Case $4$, we have from \eqref{4J33}, \eqref{J331}, \eqref{J332} and \eqref{J333} that
 \begin{align}\label{4J330}
	J_{33}(t,x,v)&\leq \frac{C_{m,k}}{N}(1+t)^{-\frac{3}{4}}\sup_{0\leq s\leq t}\|(1+s)^\frac{3}{4}w_{k}f(s)\|_{L^\infty_{x,v}}\notag\\
	&\qquad\qquad+C_{m,k,N,\la}(1+t)^{-\frac{3}{4}}(\|\nu^{-1}f_0\|_{L^2_{x,v}}+\|\nu^{-1}f_0\|_{L^2_vL^1_x}).
\end{align}

Combining the four cases \eqref{1J33}, \eqref{2J33}, \eqref{3J33} and \eqref{4J330}, we obtain
 \begin{align}\label{J33}
	J_{33}(t,x,v)&\leq C_{m,k}\big(\frac{1}{N}+\la\big)(1+t)^{-\frac{3}{4}}\sup_{0\leq s\leq t}\|(1+s)^\frac{3}{4}w_{k}f(s)\|_{L^\infty_{x,v}}\notag\\
	&\qquad+C_{m,k,N,\la}(1+t)^{-\frac{3}{4}}(\|\nu^{-1}f_0\|_{L^2_{x,v}}+\|\nu^{-1}f_0\|_{L^2_vL^1_x}).
\end{align}

We collect \eqref{milds1}, \eqref{J3}, \eqref{J310}, \eqref{J32} and \eqref{J33} to get
\begin{align*}
	\dis |(w_k f)(t,x,v)|&\leq C_{m,k}(1+t)^{-\frac{3}{4}}\|w_{k+|\ga|}f_0\|_{L^\infty_{x,v}}+C_km^{\ga+3}(1+t)^{-\frac{3}{4}}\sup_{0\leq s\leq t}\|(1+s)^\frac{3}{4}w_{k}f(s)\|_{L^\infty_{x,v}} \notag\\
	&\qquad+C_{m,k}\big(\frac{1}{N}+\la\big)(1+t)^{-\frac{3}{4}}\sup_{0\leq s\leq t}\|(1+s)^\frac{3}{4}w_{k}f(s)\|_{L^\infty_{x,v}}\notag\\
	&\qquad+C_{m,k,N,\la}(1+t)^{-\frac{3}{4}}(\|\nu^{-1}f_0\|_{L^2_{x,v}}+\|\nu^{-1}f_0\|_{L^2_vL^1_x}).
\end{align*}
Hence, \eqref{slinearLinf} follows by first choosing small $m$ and then choosing small $\la$ and large $N$.

We can prove \eqref{slinearLinfL2} in the similar way. The case that $|v|\geq N$ or $|v|\leq N$, $|\eta|\geq 2N$ or $|\eta|\leq2N$, $|\xi|\geq3N$ or $s-s_1\leq\la$ can be estimated as the first three cases above. Then using the approximation function $l_N$ which is defined in \eqref{deflN}, we deduce that
\begin{align*}
	\|(w_{k}f)(t,v)\|_{L^2_x}&\leq C_{m,k}(1+t)^{-\frac{3}{4}}\|w_{k+|\ga|}f_0\|_{L^\infty_vL^2_x}+C_km^{\ga+3}(1+t)^{-\frac{3}{4}}\sup_{0\leq s\leq t}\|(1+s)^\frac{3}{4}w_{k}f(s)\|_{L^\infty_vL^2_x} \notag\\
	&\quad+C_{m,k}\big(\frac{1}{N}+\la\big)(1+t)^{-\frac{3}{4}}\sup_{0\leq s\leq t}\|(1+s)^\frac{3}{4}w_{k}f(s)\|_{L^\infty_vL^2_x}\notag\\
	&\quad+C_{m,k,N}\int^t_0 e^{-\nu_N(t-s)}\notag\\
	&\quad\quad\times\iint_{|\eta|\leq2N,|\xi|\leq3N} \int^{s-\la}_0e^{-\nu_N(s-s_1)}\big(\int_{\R^3}\left|f(s_1,y,\xi)\right|^2dy\big)^\frac{1}{2}ds_1d\eta d\xi  ds.
\end{align*}
Applying Cauchy-Schwarz's inequality, \eqref{tdecay1} and \eqref{linearL2s}, similar calculation as in \eqref{HC221} shows that
\begin{align*}
	\|(w_{k}f)(t,v)\|_{L^2_x}&\leq C_{m,k}(1+t)^{-\frac{3}{4}}\|w_{k+|\ga|}f_0\|_{L^\infty_vL^2_x}+C_km^{\ga+3}(1+t)^{-\frac{3}{4}}\sup_{0\leq s\leq t}\|(1+s)^\frac{3}{4}w_{k}f(s)\|_{L^\infty_vL^2_x} \notag\\
	&\quad+C_{m,k}\big(\frac{1}{N}+\la\big)(1+t)^{-\frac{3}{4}}\sup_{0\leq s\leq t}\|(1+s)^\frac{3}{4}w_{k}f(s)\|_{L^\infty_vL^2_x}\notag\\
	&\quad+C_{m,k,N,\la}(1+t)^{-\frac{3}{4}}(\|\nu^{-1}f_0\|_{L^2_{x,v}}+\|\nu^{-1}f_0\|_{L^2_vL^1_x}),
\end{align*}
which yields \eqref{slinearLinfL2} by choosing small $m$, large $N$ and small $\la$. The proof of Lemma \ref{lem5.2} is complete.
\end{proof}

After establishing the decay estimates for the symmetric case, we need to study $\CK_b$. Notice that \eqref{Kb2} and \eqref{controlK} hold for $-3<\ga<0$ and $k>3$. Since the constant $M=M(\ep,k)$ chosen in \eqref{Mep} depends on $\ep$ and $k$, we may change the constant $C_k$ into $C_{\ep,k}$. Therefore, we directly have the following lemma.

\begin{lemma}%\label{leKbs}
	Let $-3<\ga<0$ and $k>3$. Let the operator $\CK_b$ be defined in \eqref{defKb} with the constant $M=M(\ep,k)$ chosen in \eqref{Mep}. Then, for any  $1\leq p\leq \infty$ and $j\geq 0$, it holds that
	\begin{align}
		&\|w_{j} \CK_bf\|_{L^\infty_vL^p_x}\leq C_{\ep,j,k}\|w_{k}f\|_{L^\infty_vL^p_x}\label{Kblps},	
	\end{align}
where $C_{\ep,j,k}$ depends only on $\ep$, $j$ and $k$.
\end{lemma}

\subsection{Estimate on $g_2$}
With these lemmas,  we  estimate $g_2$ as follows.

\begin{lemma}\label{leg2soft}
	Let $-3<\ga<0$ and $0<\ep\leq\frac{1}{2}$. There is $k_0>0$ large enough such that for any $k\geq k_0$, there is a constant $M>0$ for the decomposition $\CK=\CK_s+\sqrt{\mu}\CK_b$  such that the following estimates hold:
\begin{align}
		\|w_{k}g_2(t)\|_{L^\infty_{x,v}}&\leq C_{\ep,k}(1+t)^{-\frac{3}{4}+\ep}\left\{\|w_{k+|\ga|}g_{10}\|_{L^\infty_{x,v}}+\|w_{k+|\ga|}g_{10}\|_{L^\infty_vL^2_x}+\|w_{k+|\ga|}g_{10}\|_{L^\infty_vL^1_x}\right\}\notag\\
			&\quad+C_{\ep,k}(1+t)^{-\frac{3}{4}+\ep}\left\{\sup_{0\leq s\leq t}\|(1+s)^{\frac{3}{4}-\ep}w_kg(s)\|^2_{L^\infty_{x,v}}\right.\notag\\
			&\qquad\qquad\qquad\qquad\qquad\qquad\left.+\sup_{0\leq s\leq t}\|(1+s)^{\frac{3}{4}-\ep}w_{k}g(s)\|^2_{L^\infty_vL^2_x}\right\},
			\label{g2Linfs}
	\end{align}
	and
	\begin{align}
		\|w_{k}g_2(t)\|_{L^\infty_vL^2_x}&\leq C_{\ep,k}(1+t)^{-\frac{3}{4}+\ep}\left\{\|w_{k+|\ga|}g_{10}\|_{L^\infty_vL^2_x}+\|w_{k+|\ga|}g_{10}\|_{L^\infty_vL^1_x}\right\}\notag\\
		&\quad+C_{\ep,k}(1+t)^{-\frac{3}{4}+\ep}\left\{\sup_{0\leq s\leq t}\|(1+s)^{\frac{3}{4}-\ep}w_kg(s)\|^2_{L^\infty_{x,v}}\right.\notag\\
		&\qquad\qquad\qquad\qquad\qquad\qquad\left.+\sup_{0\leq s\leq t}\|(1+s)^{\frac{3}{4}-\ep}w_{k}g(s)\|^2_{L^\infty_vL^2_x}\right\}.\label{g2LinfL2s}
	\end{align}
\end{lemma}
\begin{proof}
		By \eqref{reg2} and \eqref{slinearLinf}, one has
	\begin{align*}
		|(w_{k}g_2)(t,x,v)|&\leq\int_0^t\|w_kU(t-s)(\CK_b g_1)(s)\|_{L^\infty_{x,v}}ds\notag\\
		&\leq C_k\int_0^t(1+t-s)^{-\frac{3}{4}}\notag\\
		&\quad\qquad\times(\|w_{k+|\ga|}\CK_b g_1(s)\|_{L^\infty_{x,v}}+\|w_{|\ga|}\CK_b g_1(s)\|_{L^2_{x,v}}+\|w_{|\ga|}\CK_b g_1(s)\|_{L^2_vL^1_x})ds.
	\end{align*}
By the fact that $\|w_kf\|_{L^2}\leq C\|w_{k+3}f\|_{L^\infty}$, it holds that
	\begin{align*}
	|(w_{k}g_2)(t,x,v)|&\leq C_{k}\int_0^t(1+t-s)^{-\frac{3}{4}}(\|w_{k+|\ga|}\CK_b g_1(s)\|_{L^\infty_{x,v}}\notag\\
	&\quad\qquad\qquad\qquad\qquad\qquad+\|w_{|\ga|+3}\CK_b g_1(s)\|_{L^\infty_vL^2_x}+\|w_{|\ga|+3}\CK_b g_1(s)\|_{L^\infty_vL^1_x})ds.
\end{align*}
We apply \eqref{Kblps} to obtain
	\begin{align}\label{g2soft1}
	|(w_{k}g_2)(t,x,v)|&\leq C_{k}\int_0^t(1+t-s)^{-\frac{3}{4}}(\|w_{k}g_1(s)\|_{L^\infty_{x,v}}+\|w_{k} g_1(s)\|_{L^\infty_vL^2_x}+\|w_{k} g_1(s)\|_{L^\infty_vL^1_x})ds.
\end{align}
Then it follows from \eqref{g2soft1}, \eqref{sg1infty}, \eqref{sg1L2} and \eqref{sg1L1} that
\begin{align*}
	|(w_{k}g_2)(t,x,v)|&\leq C_{\ep,k}\int_0^t(1+t-s)^{-\frac{3}{4}}(1+s)^{-1+\ep}\notag\\
	&\qquad\qquad\times \big(\|w_{k+|\ga|}g_{10}\|_{L^\infty_{x,v}}+\|w_{k+|\ga|} g_{10}\|_{L^\infty_vL^2_x}+\|w_{k+|\ga|}g_{10}\|_{L^\infty_vL^1_x}\big)ds\notag\\
	&\quad+C_{\ep,k}\int_0^t(1+t-s)^{-\frac{3}{4}}(1+s)^{-1+\ep}\notag\\
	&\qquad\quad\times \left\{\sup_{0\leq s\leq t}\|(1+s)^{\frac{3}{4}-\ep}w_kg(s)\|^2_{L^\infty_{x,v}}+\sup_{0\leq s\leq t}\|(1+s)^{\frac{3}{4}-\ep}w_{k}g(s)\|^2_{L^\infty_vL^2_x}\right\}ds,
\end{align*}
which yields \eqref{g2Linfs} by \eqref{tdeday}.

Similarly, by \eqref{Kblps}, \eqref{slinearLinfL2}, \eqref{sg1L2} and \eqref{sg1L1}, one has
\begin{align*}
	\|(w_{k}g_2)(t)\|_{L^\infty_vL^2_x}&\leq\int_0^t\|w_kU(t-s)(\CK_b g_1)(s)\|_{L^\infty_vL^2_x}ds\notag\\
	&\leq C_k\int_0^t(1+t-s)^{-\frac{3}{4}}\|(w_{k+|\ga|}\CK_b g_1(s)\|_{L^\infty_vL^2_x}+\|w_{|\ga|}\CK_b g_1(s)\|_{L^2_vL^1_x})ds\notag\\
	&\leq C_{\ep,k}\int_0^t(1+t-s)^{-\frac{3}{4}}(1+s)^{-1+\ep}\big(\|w_{k+|\ga|} g_{10}\|_{L^\infty_vL^2_x}+\|w_{k+|\ga|}g_{10}\|_{L^\infty_vL^1_x}\big)ds\notag\\
	&\quad+C_{\ep,k}\int_0^t(1+t-s)^{-\frac{3}{4}}(1+s)^{-1+\ep}\notag\\
	&\qquad\quad\times \left\{\sup_{0\leq s\leq t}\|(1+s)^{\frac{3}{4}-\ep}w_kg(s)\|^2_{L^\infty_{x,v}}+\sup_{0\leq s\leq t}\|(1+s)^{\frac{3}{4}-\ep}w_{k}g(s)\|^2_{L^\infty_vL^2_x}\right\}ds.
\end{align*}
Hence, \eqref{g2LinfL2s} follows from the above estimate and \eqref{tdeday}. The proof of Lemma \ref{leg2soft} is complete.
\end{proof}

\subsection{Proof of Theorem \ref{soft}}
Now we can prove Theorem \ref{soft}.
\begin{proof}[Proof of Theorem \ref{soft}]
	Recall that $g=g_1+\sqrt{\mu}g_2$. For $k\geq j\geq k_0$, a direct calculation shows that
	\begin{align*}
		\|w_kg(t)\|_{L^\infty_{x,v}}\leq \|w_kg_1(t)\|_{L^\infty_{x,v}}+\|w_{k}\sqrt{\mu}g_2(t)\|_{L^\infty_{x,v}}\leq \|w_kg_1(t)\|_{L^\infty_{x,v}}+C_{j,k}\|w_{j}g_2(t)\|_{L^\infty_{x,v}}.
	\end{align*}
	By \eqref{sg1infty}, \eqref{g2Linfs} and the condition $j\leq k$, it holds that
	\begin{align}\label{gLinfs}
		\|w_kg(t)\|_{L^\infty_{x,v}}&\leq C_{\ep,j,k}(1+t)^{-\frac{3}{4}+\ep}\left\{\|w_{k+|\ga|}g_{10}\|_{L^\infty_{x,v}}+\|w_{j+|\ga|}g_{10}\|_{L^\infty_vL^2_x}+\|w_{j+|\ga|}g_{10}\|_{L^\infty_vL^1_x}\right\}\notag\\
		+C_{\ep,j,k}&(1+t)^{-\frac{3}{4}+\ep}\left\{\sup_{0\leq s\leq t}\|(1+s)^{\frac{3}{4}-\ep}w_kg(s)\|^2_{L^\infty_{x,v}}+\sup_{0\leq s\leq t}\|(1+s)^{\frac{3}{4}-\ep}w_{j}g(s)\|^2_{L^\infty_vL^2_x}\right\}.
	\end{align}
	Also we have
	\begin{align*}
		\|w_{j}g(t)\|_{L^\infty_vL^2_x}\leq \|w_{j}g_1(t)\|_{L^\infty_vL^2_x}+\|w_{j}\sqrt{\mu}g_2(t)\|_{L^\infty_vL^2_x}\leq \|w_{j}g_1(t)\|_{L^\infty_vL^2_x}+\|w_{j}g_2(t)\|_{L^\infty_vL^2_x}.
	\end{align*}
 Then it follows from \eqref{sg1L2}, \eqref{g2LinfL2s} and the condition $j\leq k$ that
	\begin{align}\label{gLinfL2s}
		\|w_{j}g(t)\|_{L^\infty_vL^2_x}&\leq C_{\ep,j}(1+t)^{-\frac{3}{4}+\ep}\left\{\|w_{j+|\ga|}g_{10}\|_{L^\infty_vL^2_x}+\|w_{j+|\ga|}g_{10}\|_{L^\infty_vL^1_x}\right\}\notag\\
		+C_{\ep,j}&(1+t)^{-\frac{3}{4}+\ep}\left\{\sup_{0\leq s\leq t}\|(1+s)^{\frac{3}{4}-\ep}w_kg(s)\|^2_{L^\infty_{x,v}}+\sup_{0\leq s\leq t}\|(1+s)^{\frac{3}{4}-\ep}w_{j}g(s)\|^2_{L^\infty_vL^2_x}\right\}.
	\end{align}
	Combining \eqref{gLinfs} and \eqref{gLinfL2s}, we obtain
	\begin{align}\label{gsoft}
		&\sup_{0\leq s\leq t}\|(1+s)^{\frac{3}{4}-\ep}w_kg(s)\|_{L^\infty_{x,v}}+\sup_{0\leq s\leq t}\|(1+s)^{\frac{3}{4}-\ep}w_{j}g(s)\|_{L^\infty_vL^2_x}\notag\\
		&\leq C_{\ep,j,k}\left\{\|w_{k+|\ga|}g_{10}\|_{L^\infty_{x,v}}+\|w_{j+|\ga|}g_{10}\|_{L^\infty_vL^2_x}+\|w_{j+|\ga|}g_{10}\|_{L^\infty_vL^1_x}\right\}\notag\\
		&\qquad+C_{\ep,j,k}\left\{\sup_{0\leq s\leq t}\|(1+s)^{\frac{3}{4}-\ep}w_kg(s)\|^2_{L^\infty_{x,v}}+\sup_{0\leq s\leq t}\|(1+s)^{\frac{3}{4}-\ep}w_{j}g(s)\|^2_{L^\infty_vL^2_x}\right\}.
	\end{align}
	Recall our definitions of $\|\cdot\|_{X_{j,k}}$ in \eqref{defX} and $\|\cdot\|_{Y_{j,k}}$ in \eqref{defY}. By the local-in-time existence together with the continuity argument,  from \eqref{gsoft}, \eqref{SE} follows by \eqref{smallnesssoft} for a small constant $\ep_0$ which depends on $\ep$, $j$ and $k$. Hence, the global solution is established and the proof of Theorem \ref{soft} is complete.
\end{proof}

\medskip
\noindent {\bf Acknowledgments:}\, The authors thank an anonymous referee for very valuable and helpful comments on the manuscript.  RJD was partially supported by the General Research Fund (Project No.~14303321) from RGC of Hong Kong and a Direct Grant from CUHK. SQL was supported by grants from the National Natural Science Foundation of China (contract: 12325107). This work was also partially supported by the Fundamental Research Funds for the Central Universities.

\medskip

\noindent{\bf Conflict of Interest:} The authors declare that they have no conflict of interest.

%\newpage

\end{document}